\documentclass[a4paper,10pt]{amsart}
\usepackage{amscd}
\usepackage{amssymb}
\usepackage[all]{xy}
\usepackage{graphicx}
\usepackage[T1]{fontenc}
\usepackage[sc]{mathpazo}
\date{6 February 2012}
\title[MC Elements]{MC Elements in Pronilpotent DG Lie Algebras}

\author{Amnon Yekutieli}
\address{A. Yekutieli: Department of  Mathematics
Ben Gurion University,
Be'er Sheva 84105,
Israel}
\email{amyekut@math.bgu.ac.il}
\subjclass{}
\thanks{{\em Mathematics Subject Classification} 2000.
Primary: 53D55; Secondary: 17B40, 14B10, 18D05.}
\keywords{DG Lie algebras, L-infinity morphisms, Deligne groupoid,
Deformation quantization.}
\thanks{This research was supported by the Israel Science Foundation.}

\newtheorem{thm}[equation]{Theorem}
\newtheorem{cor}[equation]{Corollary}
\newtheorem{prop}[equation]{Proposition}
\newtheorem{lem}[equation]{Lemma}
\theoremstyle{definition}
\newtheorem{dfn}[equation]{Definition}
\newtheorem{rem}[equation]{Remark}
\newtheorem{exa}[equation]{Example}

\numberwithin{equation}{section}

\newcommand{\iso}{\xrightarrow{\simeq}}

\newcommand{\xar}{\xrightarrow}
\newcommand{\opn}{\operatorname}
\newcommand{\cat}[1]{\operatorname{\mathsf{#1}}}

\newcommand{\ol}{\overline}

\newcommand{\rmitem}[1]{\item[\text{\textup{(#1)}}]}
\newcommand{\mfrak}[1]{\mathfrak{#1}}
\newcommand{\mcal}[1]{\mathcal{#1}}

\newcommand{\mbf}[1]{\mathbf{#1}}
\newcommand{\mrm}[1]{\mathrm{#1}}
\newcommand{\mbb}[1]{\mathbb{#1}}

\newcommand{\smfrac}[2]{\textstyle \frac{#1}{#2}}
\newcommand{\tup}[1]{\textup{#1}}
\newcommand{\bsym}[1]{\boldsymbol{#1}}
\newcommand{\boplus}{\bigoplus\nolimits}

\newcommand{\bosum}{\sum\nolimits}

\newcommand{\ot}{\otimes}
\newcommand{\til}[1]{\tilde{#1}}
\newcommand{\what}[1]{\widehat{#1}}
\newcommand{\hatotimes}{\, \what{\otimes} \,}
\newcommand{\hot}{\hatotimes}

\newcommand{\K}{\mbb{K}}

\newcommand{\N}{\mbb{N}}
\newcommand{\Z}{\mbb{Z}}
\newcommand{\Q}{\mbb{Q}}
\newcommand{\g}{\mfrak{g}}
\newcommand{\m}{\mfrak{m}}
\newcommand{\h}{\mfrak{h}}

\newcommand{\n}{\mfrak{n}}

\renewcommand{\a}{\mfrak{a}}

\newcommand{\bwedge}{{\textstyle \bigwedge}}

\renewcommand{\d}{\mrm{d}}

\newcommand{\sg}{\sigma}
\newcommand{\om}{\omega}
\newcommand{\Om}{\Omega}
\newcommand{\ep}{\epsilon}
\newcommand{\la}{\lambda}
\newcommand{\ga}{\gamma}
\newcommand{\al}{\alpha}
\newcommand{\be}{\beta}
\renewcommand{\th}{\theta}
\newcommand{\pa}{\partial}

\newcommand{\twoto}{\Rightarrow}

\newcommand{\lb}{\linebreak}

\begin{document}

\begin{abstract}
Consider a pronilpotent DG (differential graded) Lie algebra over a field of
characteristic $0$. 
In the first part of the paper we introduce the {\em reduced Deligne groupoid}
associated to this DG Lie algebra. We prove that a DG Lie  
quasi-iso\-morphism between two such algebras induces an equivalence between the
corresponding reduced Deligne groupoids. This extends the famous result of
Goldman-Millson (attributed to Deligne) to the unbounded pronilpotent case.

In the second part of the paper we consider the {\em Deligne $2$-groupoid}. 
We show it exists under more relaxed assumptions than known before (the DG Lie
algebra is either nilpotent or of quasi quantum type). We prove that a DG Lie 
quasi-isomorphism between such DG Lie algebras induces a weak equivalence
between the corresponding Deligne $2$-groupoids. 

In the third part of the paper we prove that an L-infinity quasi-isomorphism
between pronilpotent DG Lie algebras induces a bijection between the sets of
gauge equivalence classes of Maurer-Cartan elements. This extends a result of
Kontsevich and others to the pronilpotent case.
\end{abstract}

\maketitle

\tableofcontents

\setcounter{section}{-1}
\section{Introduction}
\label{sec:Int}

Let $\K$ be a field of characteristic $0$.
By {\em parameter algebra} over $\K$ we mean a complete
local noetherian commutative $\K$-algebra $R$, with maximal ideal $\m$, such
that $R / \m = \K$. The important example is $R = \K[[\hbar]]$, the ring of 
formal power series in the variable $\hbar$.

Let $(R, \m)$ be a parameter algebra, and let 
$\g = \bigoplus_{i \in \Z}\, \g^i$ be a DG
(differential graded) Lie algebra over $\K$.
There is an induced pronilpotent DG Lie algebra 
\[ \m \hot \g = \bigoplus_{i \in \Z}\, \m \hot \g^i . \]
A solution $\om \in \m \hot \g^1$ of the Maurer-Cartan equation 
\[ \d(\om) + \smfrac{1}{2} [\om, \om] = 0 \]
is called an {\em MC element}. The set of MC elements is denoted by 
$\opn{MC}(\g, R)$. There is an action of the {\em gauge group} 
$\opn{exp}(\m \hot \g^0)$
on the set $\opn{MC}(\g, R)$, and we write
\begin{equation}
\ol{\opn{MC}}(\g, R) := \opn{MC}(\g, R) / \opn{exp}(\m \hot \g^0) ,
\end{equation}
the quotient set by this action.

The {\em Deligne groupoid} $\mbf{Del}(\g, R)$, introduced in \cite{GM}, is the
transformation \lb groupoid associated to the action of the group 
$\opn{exp}(\m \hot \g^0)$ on the set $\opn{MC}(\g, R)$. So the set 
$\pi_0 (\mbf{Del}(\g, R))$ of isomorphism classes of objects of this
groupoid equals $\ol{\opn{MC}}(\g, R)$. 
In \cite[Theorem 2.4]{GM} it was proved that if $R$ is artinian, $\g$ and $\h$
are DG Lie algebras concentrated in the degree range $[0, \infty)$
(we refer to this as the ``nonnegative nilpotent case''), and 
$\phi : \g \to \h$ is a DG Lie algebra quasi-isomorphism, then the induced
morphism of groupoids 
\[ \mbf{Del}(\phi, R) : \mbf{Del}(\g, R) \to \mbf{Del}(\h, R) \]
is an equivalence.

We introduce the {\em reduced Deligne groupoid} 
$\mbf{Del}^{\mrm{r}}(\g, R)$,
which is a certain quotient of the Deligne groupoid $\mbf{Del}(\g, R)$;
see Section \ref{sec:red-del} for details. 
If the {\em Deligne 2-groupoid} $\mbf{Del}^{2}(\g, R)$ is defined 
(see below), then 
\begin{equation} \label{eqn:109}
\mbf{Del}^{\mrm{r}}(\g, R) = 
\pi_1 (\mbf{Del}^{2}(\g, R)) .
\end{equation}
However the groupoid $\mbf{Del}^{\mrm{r}}(\g, R)$ always exists. 
This new groupoid also has the property that 
\[ \pi_0 (\mbf{Del}^{\mrm{r}}(\g, R)) = \ol{\opn{MC}}(\g, R) . \]

Here is the first main result of our paper. It is a generalization to the
unbounded pronilpotent case (i.e.\ the DG Lie algebras $\g$ and $\h$ can be
unbounded, and the parameter algebra $R$ doesn't have to be artinian) of
\cite[Theorem 2.4]{GM}. Our proof is similar to that of \cite[Theorem 2.4]{GM}:
we also use obstruction classes. 

\begin{thm} \label{thm:4}
Let $R$ be a parameter algebra over $\K$, and let $\phi : \g \to \h$ be a
DG Lie algebra quasi-isomorphism over $\K$. Then the morphism of groupoids
\[ \mbf{Del}^{\mrm{r}}(\phi, R) : \mbf{Del}^{\mrm{r}}(\g, R) 
\to \mbf{Del}^{\mrm{r}}(\h, R)  \]
is an equivalence. 
\end{thm}

This is Theorem \ref{thm:2} in the body of the paper. 

A DG Lie algebra $\g =  \bigoplus_{i \in \Z}\, \g^i$ is said to be 
of {\em quantum type} if $\g^i = 0$ for all $i < -1$. A DG Lie algebra 
$\til{\g}$ is said to be of {\em quasi quantum type} if 
there exists a DG Lie quasi-isomorphism $\til{\g} \to \g$, for some 
quantum type DG Lie algebra $\g$. 
Important examples of such DG Lie algebras are given in Example \ref{exa:102}.
In Section \ref{sec:del-2-grpd} we prove that if $R$ is artinian, or if
$\g$ is of quasi quantum type, then the {\em Deligne $2$-groupoid} 
$\mbf{Del}^2(\g, R)$ exists. The original construction (see \cite{Ge}) applied
only to the case when $R$ is artinian and $\g$ is of quantum type. 

Here is the second main result of this paper (repeated as Theorem
\ref{thm:105}):

\begin{thm} \label{thm:107}
Let $R$ be a parameter algebra, let $\g$ and $\h$ be
DG Lie algebras, and let $\phi : \g \to \h$ be a DG Lie algebra
quasi-isomorphism. Assume either of these two conditions holds\tup{:}
\begin{itemize}
\rmitem{i} $R$ is artinian.  
\rmitem{ii} $\g$ and $\h$ are of quasi quantum type.
\end{itemize}  
Then the morphism of $2$-groupoids
\[ \mbf{Del}^2(\phi, R) : \mbf{Del}^2(\g, R) \to \mbf{Del}^2(\h, R) \]
is a weak equivalence.
\end{thm}

The proof of Theorem \ref{thm:107} relies on Theorem \ref{thm:4}, via formula 
(\ref{eqn:109}).
Theorem \ref{thm:107} plays a crucial role in our new proof of {\em twisted
deformation quantization}, in the revised version of \cite{Ye4}.

An {\em  $\mrm{L}_{\infty}$ morphism} $\Phi : \g \to \h$ is a sequence 
$\Phi = \{ \phi_j \}_{j \geq 1}$ of $\K$-linear functions 
\[ \phi_j : \bwedge^j \g \to \h \]
that generalizes the notion of DG Lie algebra homomorphism 
$\phi : \g \to \h$. Thus $\phi_1 : \g \to \h$ is a 
DG Lie algebra homomorphism, up to a homotopy given by $\phi_2$; and so on. 
See Section \ref{sec:L-infty} for details. 
The morphism $\Phi$ is called an {\em $\mrm{L}_{\infty}$ quasi-isomorphism}
if $\phi_1 : \g \to \h$ is a quasi-isomorphism. 
The concept of $\mrm{L}_{\infty}$ morphism gained
prominence after the Kontsevich Formality Theorem from 1997 (see
\cite{Ko2}). 

An $\mrm{L}_{\infty}$ morphism $\Phi : \g \to \h$ induces an $R$-multilinear 
$\mrm{L}_{\infty}$ morphism 
\[ \Phi_R = \{ \phi_{R, j} \}_{j \geq 1} : \m \hot \g \to \m \hot \h . \]
Given an element $\om \in \m \hot \g^1$ we write
\begin{equation}
\opn{MC}(\Phi, R) (\om) :=
\sum_{j \geq 1} \, \smfrac{1}{j!} \phi_{R, j}
(\underset{j}{\underbrace{\om, \ldots, \om}}) \in \m \hot \h^1 .
\end{equation}
This sum converges in the $\m$-adic topology of $\m \hot \h^1$. 
It is known that the function $\opn{MC}(\Phi, R)$ sends MC elements to MC
elements, and it respects gauge equivalence (see Propositions \ref{prop:4}
and \ref{prop:5}).
So there is an induced function
$\ol{\opn{MC}}(\Phi, R)$ from 
$\ol{\opn{MC}}(\g, R)$ to $\ol{\opn{MC}}(\h, R).$

Here is the third main result of our paper.

\begin{thm} \label{thm:1}
Let $\g$ and $\h$ be DG Lie algebras, let $R$ be a parameter algebra, and let 
$\Phi : \g \to \h$ be an $\mrm{L}_{\infty}$ quasi-isomorphism, all over the
field $\K$. Then the function 
\[ \ol{\opn{MC}}(\Phi, R) : \ol{\opn{MC}}(\g, R) \to \ol{\opn{MC}}(\h, R)  \]
is bijective.
\end{thm}

This is Theorem \ref{thm:3} in the body of the paper. We emphasize
that this result is in the unbounded pronilpotent case. 
The proof of Theorem \ref{thm:1} goes like this: we use the bar-cobar
construction to reduce to the case of a DG Lie algebra quasi-isomorphism
$\til{\Phi} : \til{\g} \to \til{\h}$; and
then we use Theorem \ref{thm:4}.

Theorem \ref{thm:1} was known in the nilpotent case; see 
\cite[Theorem  4.6]{Ko2}, and  \cite[Theorem 3.6.2]{CKTB}. 
The proof sketched in \cite[Section 4.5]{Ko2} relies on the structure 
of $\mrm{L}_{\infty}$ algebras. The proof in
\cite[Section 3.7]{CKTB} relies on the work of Hinich on the Quillen model
structure of coalgebras. It is not clear whether these methods work also in the
pronilpotent case. 

In this paper we only consider pronilpotent DG Lie algebras of the form 
$\m \hot \g$. 
Presumably Theorems \ref{thm:4}, \ref{thm:107} and \ref{thm:1} can be extended
to a more general setup -- see Remark \ref{rem:11}.

\medskip \noindent
{\bf Acknowledgments.} 
I wish to thank James Stasheff, Vladimir Hinich, Michel Van den Bergh,
William Goldman, Oren Ben Bassat, Marco Manetti and Ronald Brown
for useful conversations.
Thanks also to the referee for reading the paper carefully and providing
several constructive remarks.

\section{Some Facts about DG Lie Algebras}
\label{sec:facts}

Let $\K$ be a field of characteristic $0$.
Given $\K$-modules $V, W$ we write $V \otimes W$ and 
$\opn{Hom}(V, W)$ instead of $V \otimes_{\K} W$ and $\opn{Hom}_{\K}(V, W)$,
respectively.

\begin{dfn}
By {\em parameter algebra} over $\K$ we mean a complete noetherian local
commutative $\K$-algebra $R$, with maximal ideal $\m$, such that $R / \m = \K$.
We call $\m$ a {\em parameter ideal}.
\end{dfn}

The most important example is of course
$R = \K[[\hbar]]$, where $\hbar$ is a variable, called the deformation
parameter.

Note that $R = \K \oplus \m$, so the ring $R$ can be recovered from the
nonunital algebra $\m$. For any $j \in \N$ let $R_j := R / \m^{j+1}$ and 
$\m_j := \m / \m^{j+1}$. So $R_0 = \K$, each $R_j$ is an artinian local ring
with maximal ideal $\m_j$, 
$R \cong \lim_{\leftarrow j} R_j$,
and
$\m \cong  \lim_{\leftarrow j} \m_j$.

Let us fix a parameter algebra $(R, \m)$.
Given an $R$-module $M$, its {\em $\m$-adic completion} is
\[ \what{M} := \lim_{\leftarrow j}\, (R_j \ot_R M) . \]
The module $M$ is called {\em $\m$-adically complete} if the canonical
homomorphism $M \to \what{M}$ is bijective. (Some texts would say that $M$ is
complete and separated.) Since $R$ is noetherian, the $\m$-adic completion
$\what{M}$ of any $R$-module $M$ is $\m$-adically complete; 
see \cite[Corollary 3.5]{Ye5}.

Given an $R$-module $M$ and a $\K$-module $V$ let us write 
\begin{equation} \label{eqn:52}
M \hot V := \lim_{\leftarrow j}\, (R_j \ot_R (M \ot V)) , 
\end{equation}
namely $M \hot V$ is the $\m$-adic completion of the $R$-module $M \ot V$. 
If $W$ is another $\K$-module, then there is a unique $R$-module isomorphism
\[ M \hot (V \ot W) \cong (M \hot V) \hot W \]
that commutes with the canonical homomorphisms from 
$M \ot V \ot W$; hence we shall simply denote this complete $R$-module by 
$M \hot V \hot W$.

Let $\g = \boplus_{i \in \Z} \g^i$ be a DG Lie algebra over $\K$.
(There is no finiteness assumption on $\g$.)
For any $i$ let $\m \hatotimes \g^i$ be the complete tensor product
as in (\ref{eqn:52}).
We get a pronilpotent $R$-linear DG Lie algebra 
\begin{equation} \label{eqn:64}
\m \hatotimes \g := \bigoplus_{i \in \Z}\, \m \hatotimes \g^i , 
\end{equation}
with differential $\d$ and graded Lie bracket  $[-,-]$ induced from $\g$. 

Recall that the Maurer-Cartan equation in $\m \hatotimes \g$ is
\begin{equation}
\d(\omega) + \smfrac{1}{2} [\omega, \omega] = 0
\end{equation}
for $\omega \in \m \hatotimes \g^1$.
A solution of this equation is called an {\em MC element}
of $\m \hatotimes \g$. The set of MC elements is denoted by 
$\opn{MC}(\m \hatotimes \g)$. 
In degree $0$ we have a pronilpotent Lie algebra $\m \hatotimes \g^0$,
so there is an associated pronilpotent group 
$\exp(\m \hatotimes \g^0)$, called the {\em gauge group}, and 
a bijective function 
\[ \exp : \m \hot \g^0 \to \exp(\m \hot \g^0) \]
called the exponential map.

An element $\ga \in \m \hatotimes \g^0$ acts on $\m \hatotimes \g$ by the
derivation 
\begin{equation}
\opn{ad}(\ga)(\al) := [\ga, \al] .
\end{equation}
We view $\opn{ad}(\ga)$ as an element of $R \hot \opn{End}(\g)^0$.
Let $g := \exp(\gamma) \in \exp(\m \hatotimes \g^0)$, and define 
\begin{equation} \label{eqn:11}
\opn{Ad}(g) := \exp (\opn{ad}(\gamma)) 
 = \sum_{i \geq 0} \, \smfrac{1}{i!} \opn{ad}(\ga)^i
\in R \hot \opn{End}(\g)^0 
\end{equation}
(this series converges in the $\m$-adic topology).
The element $\opn{Ad}(g)$ is an $R$-linear automorphism of the graded Lie
algebra $\m \hatotimes \g$. There is an induced group automorphism
$\opn{Ad}(g)$ of the group $\exp(\m \hatotimes \g^0)$, and this automorphism
satisfies the equation 
\begin{equation} \label{eqn:10}
\opn{Ad}(g)(h) = g \circ h \circ g^{-1} 
\end{equation}
for all $h \in \exp(\m \hatotimes \g^0)$.

There is another action of $\ga \in \m \hatotimes \g^0$ on 
$\om \in \m \hatotimes \g^{1}$:
\begin{equation}
\opn{af}(\ga)(\om) := [\ga, \om] - \d(\ga) = 
\opn{ad}(\ga)(\om) - \d(\ga) .
\end{equation}
This is an affine action, namely 
\[ \opn{af}(\ga) \in R \hot \bigl( \opn{End}(\g^1) \ltimes \g^1 \bigr) . \]
Consider the elements
$g := \exp(\ga) \in \exp(\m \hatotimes \g^0)$
and
\begin{equation}
\opn{Af}(g) := \exp(\opn{af}(\ga)) = 
\sum_{i \geq 0} \, \smfrac{1}{i!} \opn{af}(\ga)^i
\in R \hot \bigl( \opn{End}(\g^1) \ltimes \g^1 \bigr)  .
\end{equation}
(The series above converges in the $\m$-adic topology.) 
We get an affine action $\opn{Af}$ of the group $\exp(\m \hatotimes \g^0)$
on the $R$-module $\m \hatotimes \g^{1}$.
For $\om \in \m \hatotimes \g^1$ this becomes
\[ \opn{Af}(g)(\om) = 
\exp(\opn{ad}(\gamma))(\om) + 
\frac{ 1 - \exp(\opn{ad}(\gamma)) }{ \opn{ad}(\gamma) } (\d(\gamma)) . \]
The arguments in \cite[Section 1.3]{GM} (which refer to the case when
$\g^i$ are all finite dimensional over $\K$, $\g^i = 0$ for $i < 0$, and $R$ is
artinian) are valid also in our infinite case (cf.\ \cite[Section 2.2]{Ge}),
and they show that $\opn{Af}(g)$ preserves  the set 
$\opn{MC}(\m \hatotimes \g)$.
(This can be proved also using the method of Lemma \ref{lem:3}.)
We write
\begin{equation}
\ol{\opn{MC}}(\m \hatotimes \g) := 
\frac{ \opn{MC}(\m \hatotimes \g) }{ 
\opn{exp}(\m \hatotimes \g^0) } \ ,
\end{equation}
the quotient set by this action.
Given a homomorphism of DG Lie algebras 
 $\phi : \g \to \h$, and homomorphism of parameter algebras
$f : (R, \m) \to (S, \n)$, there is an induced function 
\begin{equation}
\ol{\opn{MC}}(\phi \ot f) :
\ol{\opn{MC}}(\m \hatotimes \g) \to
\ol{\opn{MC}}(\n \hatotimes \h) .
\end{equation} 

We shall need the following sort of algebraic differential calculus
(which is used a lot implicitly in deformation theory).
Let $\K[t]$ be the polynomial algebra in a variable $t$, and let $M$ be a
$\K$-module. A polynomial 
$f(t) \in \K[t] \otimes M$ defines a 
function $f : \K \to M$, namely for any $\la \in \K$ the element
$f(\la) \in M$ is gotten by substitution $t \mapsto \la$.
We refer to $f : \K \to M$ as a polynomial function, or as 
a polynomial path in $M$. 

Let $\K[\ep] := \K[t] / (t^2)$, where $\ep$ is the class of $t$. 
Given $f(t) \in \K[t] \otimes M$ and $\la \in \K$ we denote by 
$f(\la + \ep) \in \K[\ep] \otimes M$ the result of the substitution 
$t \mapsto \la + \ep$.

\begin{lem} \label{lem:7}
Let $f(t) \in \K[t] \otimes M$. If 
\[ f(\la + \ep)  = f(\la) \]
in $\K[\ep] \otimes M$ for all $\la \in \K$, then $f(\la) = f(0)$
for all $\la \in \K$.
\end{lem}

\begin{proof}
This is an elementary calculation; note that $\K$ has characteristic $0$.
\end{proof}

Given an element $\om \in \opn{MC}(\m \hatotimes \g)$, consider the 
$R$-linear operator (of degree $1$)
\begin{equation} \label{eqn:17}
 \d_{\om} := \d + \opn{ad}(\om) 
\end{equation}
on $\m \hatotimes \g$.

\begin{lem} \label{lem:3}
Let $\om, \om' \in \opn{MC}(\m \hatotimes \g)$,
and let $g \in \opn{exp}(\m \otimes \g^0)$ be such that
$\om' = \opn{Af}(g)(\om)$. 
Then for any $i \in \Z$ the diagram of $R$-modules
\[ \UseTips \xymatrix @C=10ex @R=6ex {
\m \hatotimes \g^{i}
\ar[r]^{\opn{Ad}(g)}
\ar[d]_{\d_{\om}}
&
\m \hatotimes \g^{i}
\ar[d]^{\d_{\om'}}
\\
\m \hatotimes \g^{i+1}
\ar[r]^{\opn{Ad}(g)}
&
\m \hatotimes \g^{i+1}
} \]
is commutative.
\end{lem}

\begin{proof}
Since $\m \hatotimes \g^{i}$ and $\m \hatotimes \g^{i+1}$ are
$\m$-adically complete $R$-modules, it suffices to verify this after replacing
$R$ with $R_j$. Therefore we can assume that $R$ is artinian.

Consider the DG Lie algebra 
$\K[t] \otimes \m \otimes \g$. 
Let $\ga := \log(g) \in \m \otimes \g^0$, and define 
\[ g(t) := \exp(t \ga) \in \exp(\K[t] \otimes \m \otimes \g^0)
\subset \K[t] \otimes R \otimes \opn{End}(\g^0) . \]
So $g(0) = 1$ and $g(1) = g$. Next let 
\[ \om(t) := \opn{Af}(g(t))(\om) \in \K[t] \otimes \m \otimes \g^1 , \]
which is an MC element of $\K[t] \otimes \m \otimes \g$, and it 
satisfies $\om(0) = \om$ and $\om(1) = \om'$. Consider the polynomial 
\[ f(t) := \opn{Ad}(g(1-t)) \circ \d_{\om(t)} \circ \opn{Ad}(g(t)) \in 
\K[t] \otimes R \otimes \opn{Hom}(\g^{i}, \g^{i+1}) . \]
It satisfies  
\[ f(0) = \opn{Ad}(g) \circ \d_{\om} \]
and 
\[ f(1) =  \d_{\om'} \circ \opn{Ad}(g)  . \]
We will prove that $f$ is constant.
See diagram below depicting $f(\la)$, $\la \in \K$.

\[ \UseTips \xymatrix @C=10ex @R=7ex {
\m \otimes \g^{i}
\ar[r]^{\opn{Ad}(g(\la))}
\ar@{-->}[d]
&
\m \otimes \g^{i}
\ar[d]^{\d_{\om(\la)}}
\ar@{-->}[r]
&
\m \otimes \g^{i}
\ar@{-->}[d]
\\
\m \otimes \g^{i+1}
\ar@{-->}[r]
&
\m \otimes \g^{i+1}
\ar[r]^{\opn{Ad}(g(1 - \la))}
&
\m \otimes \g^{i+1}
} \]

\medskip
Take any $\la \in \K$. Then 
\[ \begin{aligned}
& f(\la +\ep) - f(\la) = \\
& \qquad \opn{Ad}(g(1 - \ep - \la)) \circ
\bigl( \d_{\om(\la + \ep)} \circ \opn{Ad}(g(\ep)) -  
\opn{Ad}(g(\ep)) \circ \d_{\om(\la)} \bigr) \circ \opn{Ad}(g(\la)) 
\end{aligned} \]
in $\K[\ep] \otimes R \otimes \opn{Hom}(\g^{i}, \g^{i+1})$.
A calculation shows that 
\[ \opn{Ad}(g(\ep)) = 1 + \ep \cdot \opn{ad}(\ga) \in 
\K[\ep] \otimes R \otimes \opn{End}(\g) \]
and
\[ \om(\la + \ep) = \om(\la) + \ep \cdot [\ga, \om(\la)] - 
\ep \cdot \d(\ga) \in \K[\ep] \otimes \m \otimes \g^1 . \]
Hence for any 
$\al \in R \otimes \g^{-1}$ we have
\[ 
\begin{aligned}
& \bigl( \d_{\om(\la + \ep)} \circ \opn{Ad}(g(\ep)) \bigr)(\al) = \\
& \qquad \d(\al) + \ep \cdot \d([\ga, \al]) 
+ [\om(\la + \ep), \al] + \ep [ \om(\la + \ep), [\ga, \al]]
\end{aligned} \]
and 
\[ \begin{aligned}
& \bigl( \opn{Ad}(g(\ep)) \circ \d_{\om(\la} \bigr)(\al) = \\
& \qquad \d(\al) + [\om(\la), \al] + \ep [\ga, \d(\al)] +
\ep [\ga, [\om(\la), \al ]] \ .
\end{aligned} \]
After expanding terms and using the graded Jacobi identity we see that 
\[ \bigl( \d_{\om(\la + \ep)} \circ \opn{Ad}(g(\ep)) \bigr)(\al) =
\bigl( \opn{Ad}(g(\ep)) \circ \d_{\om(\la} \bigr)(\al) \]
in $\K[\ep] \otimes R \otimes \g^0$.
Therefore $f(\la +\ep) = f(\la)$. By Lemma \ref{lem:7} we conclude that 
$f$ is constant.
\end{proof}

\begin{prop} \label{prop:100}
\begin{enumerate}
\item Let $\om \in \opn{MC}(\m \hot \g)$. Then 
$\d_{\om}$ is a degree $1$ derivation of the graded Lie algebra $\m \hot \g$,
and $\d_{\om} \circ \d_{\om} = 0$. We obtain a new DG Lie algebra 
$(\m \hot \g)_{\om}$, with the same Lie bracket $[-,-]$, and a new differential
$\d_{\om}$.

\item Let $g \in \opn{exp}(\m \hot \g^0)$, and let 
$\om' := \opn{Af}(g)(\om) \in \opn{MC}(\m \hot \g)$. Then 
\[ \opn{Ad}(g) : (\m \hot \g)_{\om} \to (\m \hot \g)_{\om'} \]
is an isomorphism of DG Lie algebras.
\end{enumerate}
\end{prop}

\begin{proof}
(1) This is well known (and very easy to check). 

\medskip \noindent
(2) Let $\ga := \log(g) \in \m \hot \g^0$.
Since $\opn{ad}(\ga)$ is a derivation of the graded Lie algebra 
$\m \hot \g$, it follows that 
$\opn{Ad}(g) = \exp(\opn{ad}(\ga))$ is an automorphism of 
$\m \hot \g$. By Lemma \ref{lem:3} this automorphism exchanges $\d_{\om}$
\and $\d_{\om'}$.
\end{proof}

\section{The Reduced Deligne Groupoid}
\label{sec:red-del}

As before, $\K$ is a field of characteristic $0$,
$(R, \m)$ is a parameter algebra over $\K$, and 
$\g = \boplus_{i \in \Z} \g^i$ is a DG Lie algebra over $\K$. 

Let us write 
\begin{equation} \label{eqn:44}
\opn{MC}(\g, R) := \opn{MC}(R \hatotimes \g)
\end{equation}
and 
\begin{equation} \label{eqn:45}
\opn{G}(\g, R) := \exp(\m \hatotimes \g^0) .
\end{equation} 
Given $\om, \om' \in \opn{MC}(\g, R)$, let
\begin{equation} \label{eqn:43}
\opn{G}(\g, R)(\om, \om') :=
\{ g \in \opn{G}(\g, R) \mid \opn{Af}(g)(\om) = \om' \} .
\end{equation}
As in \cite{GM} we define the {\em Deligne groupoid} 
$\mbf{Del}(\g, R)$ to be the transformation \lb groupoid
associated to the action of the gauge group $\opn{G}(\g, R)$ on
the set $\opn{MC}(\g, R)$. 
So the set of objects of $\mbf{Del}(\g, R)$ is $\opn{MC}(\g, R)$, and the set
of morphisms $\om \to \om'$ in this groupoid is 
$\opn{G}(\g, R)(\om, \om')$. Identity morphisms and composition in the
groupoid are those of the group $\opn{G}(\g, R)$. 

Now suppose $\om, \om' \in \opn{MC}(\g, R)$ and 
$g \in \opn{G}(\g, R)(\om, \om')$. 
Since  
\[ \opn{Af}(g \circ h \circ g^{-1}) = 
\opn{Af}(g) \circ \opn{Af}(h) \circ \opn{Af}(g)^{-1}  \]
for any $h$, and in view of (\ref{eqn:10}), there is a group isomorphism
\begin{equation}
\opn{Ad}(g) : \opn{G}(\g, R)(\om, \om) \to \opn{G}(\g, R)(\om', \om') .
\end{equation}

Given $\om \in \opn{MC}(\g, R)$ there is the derivation 
$\d_{\om}$ of formula (\ref{eqn:17}). Let us define
\begin{equation} \label{eqn:14}
\a^{\mrm{r}}_{\om} := \opn{Im} \bigl( \d_{\om} : 
\m \hatotimes \g^{-1} \to  \m \hatotimes \g^0 \bigr)
\end{equation}
and 
\begin{equation} \label{eqn:15}
(\m \hatotimes \g^0)(\om) := \opn{Ker} \bigl( \d_{\om} : 
\m \hatotimes \g^{0} \to  \m \hatotimes \g^1 \bigr) . 
\end{equation}
These are $R$-submodules of $\m \hatotimes \g^0$.

\begin{lem} \label{lem:4}
Let $\om \in \opn{MC}(\g, R)$.
Consider the bijection of sets
\[ \exp : \m \hatotimes \g^0 \iso \exp(\m \hatotimes \g^0) = 
\opn{G}(\g, R) . \]
\begin{enumerate}
\item The module 
$(\m \hatotimes \g^0)(\om)$ is a Lie subalgebra of 
$\m \hatotimes \g^0$, and 
\[ \exp \bigl( (\m \hatotimes \g^0)(\om) \bigr) = 
\opn{G}(\g, R)(\om, \om)  \]
as subsets of $\opn{G}(\g, R)$.

\item The module $\a^{\mrm{r}}_{\om}$ is a Lie ideal of the Lie algebra 
$(\m \hatotimes \g^0)(\om)$, and the subset
\[ N^{\mrm{r}}_{\om} = N^{\mrm{r}}(\g, R)_{\om} :=
\exp(\a^{\mrm{r}}_{\om}) \]
is a normal subgroup of $\opn{G}(\g, R)(\om, \om)$.

\item Let $g \in \opn{G}(\g, R)$ and
$\om' := \opn{Af}(g)(\om)$. Then 
\[ \opn{Ad}(g) \bigl( N^{\mrm{r}}_{\om} \bigr) = N^{\mrm{r}}_{\om'} . \]
\end{enumerate}
\end{lem}

\begin{proof}
(1) Since $\d_{\om}$ is a graded derivation of the  graded Lie
algebra $\m \hatotimes \g$, its kernel is a graded Lie subalgebra, and its image
is a graded Lie ideal in the kernel. In degree $0$ we get a Lie subalgebra 
$(\m \hatotimes \g^0)(\om)$, and a Lie ideal $\a^{\mrm{r}}_{\om}$ in it.

Because $\d_{\om}$ is a continuous homomorphism between complete $R$-modules,
its kernel $(\m \hatotimes \g^0)(\om)$ is closed; so this is a closed Lie
subalgebra of $\m \hatotimes \g^0$. This implies that the subset 
$\exp \bigl( (\m \hatotimes \g^0)(\om) \bigr)$ is a closed subgroup of 
$\opn{G}(\g, R)$. 
Moreover, let $\ga \in \m \hot \g^0$ and $g := \exp(\ga)$.  
In the proof of  \cite[Theorem 2.2]{Ge} it is shown that 
$\d_{\om}(\gamma) = 0$ iff 
$\opn{Af}(g)(\om) = \om$. This shows that  
\[ \exp \bigl( (\m \hatotimes \g^0)(\om) \bigr) = 
\opn{G}(\g, R)(\om, \om) . \]

\medskip \noindent
(2) We already know that $\a^{\mrm{r}}_{\om}$ is a Lie ideal of 
$(\m \hatotimes \g^0)(\om)$; but since this is not a closed ideal in general,
it is not immediate that the subset $\exp(\a^{\mrm{r}}_{\om})$ is a normal
subgroup of 
$\exp \bigl( (\m \hatotimes \g^0)(\om) \bigr)$.

Consider the CBH series
\[ F(x_1, x_2) = \sum_{j \geq 1}\, F_j(x_1, x_2) , \]
where $F_j(x_1, x_2)$ are homogeneous elements of degree $i$ in the 
free Lie algebra in the variables $x_1, x_2$ over $\Q$. 
It is known (cf.\ \cite{Bo}) that 
\begin{equation} \label{eqn:53}
\exp(\ga_1) \cdot \exp(\ga_2)  = \exp(F(\ga_1, \ga_2))
\end{equation}
for $\ga_1, \ga_2 \in \m \hot \g^0$. 

Let us define a bracket $[-,-]_{\om}$ on $\m \hot \g^{-1}$ as follows:
\[ [\al_1, \al_2]_{\om} := [ \d_{\om}(\al_1), \al_2] . \]
In general this is not a Lie bracket (the Jacobi identity may fail). However,
since $\d_{\om}$ is a square zero derivation of
the graded Lie algebra $\m \hot \g$, we have 
\[ \d_{\om}([\al_1, \al_2]_{\om}) = [\d_{\om}(\al_1), \d_{\om}(\al_2)] . \]

For any $j \geq 1$ and $\al_1, \al_2 \in \m \hot \g^{-1}$  consider 
the element 
$F_{j, \om}(\al_1, \al_2) \in \m \hot \g^{-1}$
gotten by evaluating the Lie polynomial $F_j(x_1, x_2)$
at $x_i \mapsto \al_i$, using the bracket $[-,-]_{\om}$. 
Now take any $\ga_1, \ga_2 \in \a^{\mrm{r}}_{\om}$, and choose 
$\al_1, \al_2 \in \m \hot \g^{-1}$ such that $\ga_i = \d_{\om}(\al_i)$. Then 
\[ \d_{\om}(F_{j, \om}(\al_1, \al_2)) = 
F_j(\ga_1, \ga_2) \in \m \hot \g^0 . \]
Let 
\[ \al :=  \bosum_{j \geq 1}\,  F_{j, \om}(\al_1, \al_2) \in 
\m \hot \g^{-1}  \]
and $\ga := \d_{\om} (\al) \in \a^{\mrm{r}}_{\om}$. 
By continuity we get
$F(\ga_1, \ga_2) = \ga$. 
{}From this and formula (\ref{eqn:53}) we see that 
$N^{\mrm{r}}_{\om} = \exp(\a^{\mrm{r}}_{\om})$ is a
subgroup of $\opn{G}(\g, R) = \exp (\m \hatotimes \g^0)$.

Similarly, Lemma \ref{lem:3}(2) implies that the subset
$\exp(\a^{\mrm{r}}_{\om})$ is
invariant under the operations $\opn{Ad}(g)$, 
$g \in \opn{G}(\g, R)(\om, \om)$. Therefore $N^{\mrm{r}}_{\om}$ is a
normal subgroup of $\opn{G}(\g, R)(\om, \om)$.

(3) According to Lemma \ref{lem:3}(2) we know that 
$\opn{Ad}(g)(\a^{\mrm{r}}_{\om}) = \a^{\mrm{r}}_{\om'}$.
\end{proof}

Note that the set 
$\opn{G}(\g, R)(\om, \om')$ has a left action by the group 
$N^{\mrm{r}}_{\om'}$, and a right action by the group
$N^{\mrm{r}}_{\om}$. Define
\begin{equation} \label{eqn:16}
\opn{G}^{\mrm{r}}(\g, R)(\om, \om') := 
\opn{G}(\g, R)(\om, \om') / N^{\mrm{r}}_{\om} \, ,
\end{equation}
the quotient set. So there is a surjective function 
\begin{equation} \label{eqn:51}
\eta_1 : \opn{G}(\g, R)(\om, \om') \to \opn{G}^{\mrm{r}}(\g, R)(\om, \om') . 
\end{equation}

By Lemma \ref{lem:4}(3), the multiplication map of 
$\opn{G}(\g, R)$ induces maps
\begin{equation} \label{eqn:68}
\opn{G}^{\mrm{r}}(\g, R)(\om, \om') \times 
\opn{G}^{\mrm{r}}(\g, R)(\om', \om'') \to
\opn{G}^{\mrm{r}}(\g, R)(\om, \om'')
\end{equation}
for any $\om, \om', \om'' \in \opn{MC}(\g, R)$.
If $\om' = \om$ then 
$\opn{G}^{\mrm{r}}(\g, R)(\om, \om)$ is a group.

\begin{dfn}
The {\em reduced Deligne groupoid} associated to $\g$ and $(R, \m)$
is the \lb groupoid $\mbf{Del}^{\mrm{r}}(\g, R)$ defined as follows. 
The set of objects of this groupoid is $\opn{MC}(\g, R)$. For any 
$\om, \om' \in \opn{MC}(\g, R)$, the set of morphisms $\om \to \om'$ is 
the set $\opn{G}^{\mrm{r}}(\g, R)(\om, \om')$
from formula (\ref{eqn:16}). The composition in $\mbf{Del}^{\mrm{r}}(\g, R)$
is given by formula (\ref{eqn:68}), and the identity morphisms are those of the
groups $\opn{G}^{\mrm{r}}(\g, R)(\om, \om)$.
\end{dfn}

There is a morphism of groupoids (i.e.\ a functor)
\begin{equation} \label{eqn:13}
\bsym{\eta} = (\eta_0, \eta_1) : \mbf{Del}(\phi, R) \to
\mbf{Del}^{\mrm{r}}(\phi, R) ,
\end{equation}
where $\eta_0$ is the identity on the set of objects
$\opn{MC}(\g, R)$, and $\eta_1$ is the surjective function in formula
(\ref{eqn:51}).  Hence 
\begin{equation} \label{eqn:5}
\pi_0 \bigl( \mbf{Del}^{\mrm{r}}(\g, R) \bigr) =
\pi_0 \bigl( \mbf{Del}(\g, R) \bigr) =
\ol{\opn{MC}}(\m \hatotimes \g)  ,
\end{equation}
where $\pi_0(-)$ denotes the set of isomorphism classes of objects of a
groupoid.

Given a homomorphism $f : (R, \m) \to (S, \n)$ of parameter algebras, and a
homomorphism 
$\phi : \g \to \h$ of DG Lie algebras, there is an induced 
DG Lie algebra homomorphism 
\[ f \otimes \phi : \m \hatotimes \g \to \n \hatotimes \h . \]
Hence there are induced morphisms of groupoids 
$\mbf{Del}(\phi, f)$ and $\mbf{Del}^{\mrm{r}}(\phi, f)$
such that the diagram 
\[ \UseTips \xymatrix @C=11ex @R=6ex {
\mbf{Del}(\g, R)
\ar[r]^{\mbf{Del}(\phi, f)}
\ar[d]_{\bsym{\eta}}
&
\mbf{Del}(\h, S)
\ar[d]^{\bsym{\eta}}
\\
\mbf{Del}^{\mrm{r}}(\g, R)
\ar[r]^{\mbf{Del}^{\mrm{r}}(\phi, f)}
&
\mbf{Del}^{\mrm{r}}(\h, S) 
} \]
is commutative. And there is an induced function 
\[ \pi_0 \bigl( \mbf{Del}^{\mrm{r}}(\phi, f) \bigr) :
\pi_0 \bigl( \mbf{Del}^{\mrm{r}}(\g, R) \bigr) \to 
\pi_0 \bigl( \mbf{Del}^{\mrm{r}}(\h, S) \bigr) . \]
Under the equality of sets (\ref{eqn:5}), and the corresponding one for 
$\h$ and $S$, we have equality of functions
\begin{equation} \label{eqn:18}
\pi_0 \bigl( \mbf{Del}^{\mrm{r}}(\phi, f) \bigr) =
\pi_0 \bigl( \mbf{Del}(\phi, f) \bigr) =
\ol{\opn{MC}}(f \otimes \phi) . 
\end{equation}

\begin{prop} \label{prop:2}
Let $\om \in \opn{MC}(\g, R)$. 
The bijection 
\[ \exp : \m \hot \g^0 \to  \opn{G}(\g, R) \] 
induces a bijection \tup{(}of sets\tup{)}
\[ \exp : \mrm{H}^0( (\m \hot \g)_{\om} ) \to 
\opn{G}^{\mrm{r}}(\g, R)(\om, \om) . \]
This bijection is functorial w.r.t.\ homomorphisms $(R, \m) \to (R, \n)$ of
parameter algebras and homomorphisms $\g \to \h$ of DG Lie algebras.
\end{prop}

\begin{proof}
This is an immediate consequence of Lemma \ref{lem:4}(1,2) and formula 
(\ref{eqn:16}).
\end{proof}

\begin{rem}
After writing an earlier version of this paper, we were told by M. Manetti that
M. Kontsevich had mentioned the idea of a reduced Deligne groupoid already in
1994. See \cite[page 19]{Ko1} and \cite{Mt}.
\end{rem}

\section{DG Lie Quasi-isomorphisms -- Nilpotent Algebras}
\label{sec:dg-quasi-nilp}

In this section we prove several lemmas that will be used in Section 
\ref{sec:dg-quasi-comp}. We assume that 
$(R, \m)$ is an artinian parameter algebra, but $\m \neq 0$. Also we have a DG
Lie algebra quasi-isomorphism $\phi : \g \to \h$. 

Let 
\[ l(R) := \min\, \{ l \in \N \mid \m^{l+1} = 0 \} , \]
and define $\n := \m^{l(R)}$. Thus $\n$ is an ideal in $R$ satisfying 
$\m \n = 0$. Let 
$\bar{R} := R / \n$ and $\bar{\m} := \m / \n$. 
So $(\bar{R}, \bar{\m})$ is a parameter algebra, and there is a 
canonical surjection $p : R \to \bar{R}$.

Our assumption that $\m \neq 0$ implies that $l(R) \geq 1$, and that
$l(\bar{R}) < l(R)$.

The homomorphism $p : R \to \bar{R}$ induces a surjective homomorphism of DG Lie
algebras 
\[ p : \m \otimes \g \to \bar{\m} \otimes \g , \]
and likewise for $\h$. Thus we get a commutative diagram of morphisms of \lb
groupoids 
\[ \UseTips \xymatrix @C=9ex @R=6ex {
\mbf{Del}^{\mrm{r}}(\g, R)
\ar[r]^{\phi}
\ar[d]_{p}
&
\mbf{Del}^{\mrm{r}}(\h, R)
\ar[d]^{p}
\\
\mbf{Del}^{\mrm{r}}(\g, \bar{R})
\ar[r]^{\phi}
&
\mbf{Del}^{\mrm{r}}(\h, \bar{R}) \ ,
} \]
where, for the sake of brevity, we write $p$ instead of 
$\mbf{Del}^{\mrm{r}}(\g, p)$, etc. 

Given elements $\om, \om' \in \opn{MC}(\g, R)$, let
$\bar{\om} := p(\om)$ and $\bar{\om}' := p(\om')$ in 
$\opn{MC}(\g, \bar{R})$. For any element 
$\bar{g} \in \opn{G}(\g, \bar{R})(\bar{\om}, \bar{\om}')$
we define 
\begin{equation} \label{eqn:4}
\opn{G}(\g, R)(\om, \om') / \bar{g} :=  
\{ g \in \opn{G}(\g, R)(\om, \om') \mid p(g) = \bar{g} \} .
\end{equation}

Next, given an element $\bar{\om} \in \opn{MC}(\g, \bar{R})$, 
let us denote by $\mbf{Del}(\g, R) / \bar{\om}$ 
the fiber over $\bar{\om}$ of the morphism of groupoids 
\[ p : \mbf{Del}(\g, R) \to \mbf{Del}(\g, \bar{R}) . \]
Thus the set of objects of  $\mbf{Del}(\g, R) / \bar{\om}$ is the set 
\begin{equation}
\opn{MC}(\g, R) / \bar{\om} := \{ \om \in \opn{MC}(\g, R) \mid 
p(\om) = \bar{\om} \} .
\end{equation}
The set of morphisms $\om \to \om'$ in $\mbf{Del}(\g, R) / \bar{\om}$ 
is 
\begin{equation}
\opn{G}(\g, R)(\om, \om') / 1 :=  
\{ g \in \opn{G}(\g, R)(\om, \om') \mid p(g) = 1 \} . 
\end{equation}

We shall need some of this construction also for the groupoid 
$\mbf{Del}^{\mrm{r}}(\g, R)$.
Given elements $\om, \om' \in \opn{MC}(\g; R)$, let
$\bar{\om} := p(\om)$ and $\bar{\om}' := p(\om')$ in 
$\opn{MC}(\g, \bar{R})$.
Suppose 
$\bar{g} \in \opn{G}^{\mrm{r}}(\g, \bar{R})(\bar{\om}, \bar{\om}')$.
We define the subset 
\begin{equation}
\opn{G}^{\mrm{r}}(\g, R)(\om, \om') / \bar{g} := 
\{ g \in \opn{G}^{\mrm{r}}(\g, R)(\om, \om') \mid p(g) = \bar{g} \} . 
\end{equation}

We now recall the obstruction functions $o_2$ and $o_1$ introduced in
\cite[Section 2.6]{GM}. Let us denote by 
$\mrm{Z}^i(\g)$ the $\K$-module of $i$-cocycles in $\g$. 
For $\al \in \m \otimes \mrm{Z}^i(\g)$ we shall denote its cohomology class by 
$[\al] \in \m \otimes \mrm{H}^i(\g) \cong \mrm{H}^i(\m \ot \g)$.

Let 
\[ \opn{cur} : \m \otimes \g^1 \to \m \otimes \g^2 \]
be the function 
\begin{equation} \label{eqn:36}
\opn{cur}(\om) := \d(\omega) + \smfrac{1}{2} [\omega, \omega] .
\end{equation}
(``$\opn{cur}$'' stands for ``curvature''.)
Thus $\om$ is an MC element iff $\opn{cur}(\om) = 0$.
Given $\bar{\om} \in \opn{MC}(\g, \bar{R})$, 
choose any lift to an element 
$\om \in \m \otimes \g^1$.
Then 
$\opn{cur}(\om) \in \n \otimes \mrm{Z}^2(\g)$,
and we define 
\begin{equation} \label{eqn:41}
o_2(\bar{\om}) := [\opn{cur}(\om)] \in \n \otimes \mrm{H}^2(\g) .
\end{equation}
It is shown in \cite{GM} that $o_2(\bar{\om})$
is independent of the choice, and the resulting obstruction function
\[ o_2 : \opn{MC}(\g, \bar{R}) \to 
\n \otimes \mrm{H}^2 (\g) \]
has the property that an element 
$\bar{\om} \in \opn{MC}(\g, \bar{R})$ 
lifts to an element of $\opn{MC}(\g, R)$
iff $o_2(\bar{\om}) = 0$.

Consider an element $\bar{\om} \in \opn{MC}(\g, \bar{R})$.
The set 
$\opn{MC}(\g, R) / \bar{\om}$, if it is nonempty, 
has a simply transitive action by the additive group 
$\n \otimes \mrm{Z}^1(\g)$, namely 
$\om \mapsto \om + \beta$ for $\beta \in \n \otimes \mrm{Z}^1(\g)$.
Given $\om, \om' \in \opn{MC}(\g, R) / \bar{\om}$, define 
\begin{equation} \label{eqn:42}
o_1(\om, \om') := [\om - \om'] \in \n \otimes \mrm{H}^1 (\g) .
\end{equation}
The obstruction function
\[ o_1 : \opn{MC}(\g, R) / \bar{\om} \ \times \ 
\opn{MC}(\g, R) / \bar{\om} \to 
\n \otimes \mrm{H}^1 (\g)  \]
has the property that the set 
$\opn{G}(\g, R)(\om, \om') / 1$ 
is nonempty  iff 
$o_1(\om, \om') = 0$. 

The obstruction functions $o_2$ and $o_1$ are functorial in $\g$ (in the
obvious sense).

\begin{rem}
It is possible to define the obstruction $o_0$ here too, but we will not use it.
Consider the exact sequence of complexes 
\[ 0 \to \n \otimes \g \to (\m \otimes \g)_{\om} \xar{ \ p \ }
(\bar{\m} \otimes \g)_{\bar{\om}} \to 0 . \]
{}From the cohomology  exact sequence we get a homomorphism
\[ \mrm{H}^{-1} ((\bar{\m} \otimes \g)_{\bar{\om}}) \to
\n \otimes \mrm{H}^{0}(\g) . \]
For 
$g, g' \in \opn{G}^{\mrm{r}}(\g, R)(\om, \om') / \bar{g}$,
the obstruction class $o_0(g, g')$ lives in the cokernel of this homomorphism.
\end{rem}

\begin{lem} \label{lem:1}
Let 
$\bar{\chi} \in \opn{MC}(\h, \bar{R})$,
$\bar{\om} \in \opn{MC}(\g, \bar{R})$,
$\chi \in \opn{MC}(\h, R) / \bar{\chi}$ 
and
\[ \bar{h} \in \opn{G}(\h, \bar{R})(\phi(\bar{\om}), \bar{\chi} ) . \]
Then there exist 
\[ \om \in \opn{MC}(\g, R) / \bar{\om} \]
and 
\[ h \in \opn{G}(\h, R)(\phi(\om), \chi) / \bar{h} . \]
\end{lem}

\begin{proof}
The proof is very similar to the proof of ``Surjective on isomorphism \lb
classes'' in \cite[Subsection 2.11]{GM}. It is illustrated in Figure
\ref{fig:1}.

Let 
\[ \bar{\chi}' := \opn{Af}(\bar{h})^{-1}(\bar{\chi}) 
=\phi(\bar{\om}) \in \opn{MC}(\h, \bar{R}) . \]
Choose any $h' \in \opn{G}(\h, R)$ lying above 
$\bar{h}$, and let 
\[ \chi' := \opn{Af}(h')^{-1}(\chi) \in \opn{MC}(\h, R) / \bar{\chi}' . \]
Since $\chi'$ exists, the obstruction class 
$o_2(\bar{\chi}')$ is zero. Now 
$\phi(\bar{\om}) = \bar{\chi}'$, so by functoriality of the obstruction classes
we get 
\[ \mrm{H}^2 (\phi)(o_2(\bar{\om})) = o_2(\bar{\chi}') = 0 . \]
The assumption is that $\mrm{H}^2 (\phi)$ is injective; 
hence $o_2(\bar{\om}) = 0$, and we can find 
$\om'' \in \opn{MC}(\g, R)$ lying above $\bar{\om}$. 

Let $\chi'' := \phi(\om'') \in \opn{MC}(\h, R) / \bar{\chi}'$. 
Consider the pair of elements 
$\chi'', \chi' \in \opn{MC}(\h, R) / \bar{\chi}'$.
There is an obstruction class 
\[ o_1(\chi'', \chi') \in \n \otimes \mrm{H}^1 (\h) . \]
By assumption the homomorphism $\mrm{H}^1 (\phi)$ is surjective, so there is 
is a cohomology class 
$c \in \n \otimes \mrm{H}^1 (\g)$
such that 
$\mrm{H}^1 (\phi)(c) = o_1(\chi'', \chi')$. 
Let 
$\gamma \in \n \otimes \mrm{Z}^1 (\g)$
be a cocycle representing $c$, and define
\[ \om := \om'' - \gamma \in \m \otimes \g^1 . \]
Then 
$\om \in \opn{MC}(\g, R) / \bar{\om}$
(it is an easy calculation done in \cite{GM}).
Let 
\[ \chi''' := \phi(\om) \in \opn{MC}(\h, R) / \bar{\chi}' . \]
Now $o_1(\om'', \om') = c$, so
\[ o_1(\chi''', \chi') = o_1(\chi'', \chi') - o_1(\chi'', \chi''') = 
o_1(\chi'', \chi') - \mrm{H}^1 (\phi)(o_1(\om'', \om')) = 0 . \]
Therefore there exists 
$h''' \in \opn{G}(\g, R)(\chi''', \chi') / 1$. 
And we have
\[ h := h' \cdot h''' \in \opn{G}(\g, R)(\chi''', \chi) / \bar{h} . \]
\end{proof}

\begin{figure}
\includegraphics[scale=0.35]{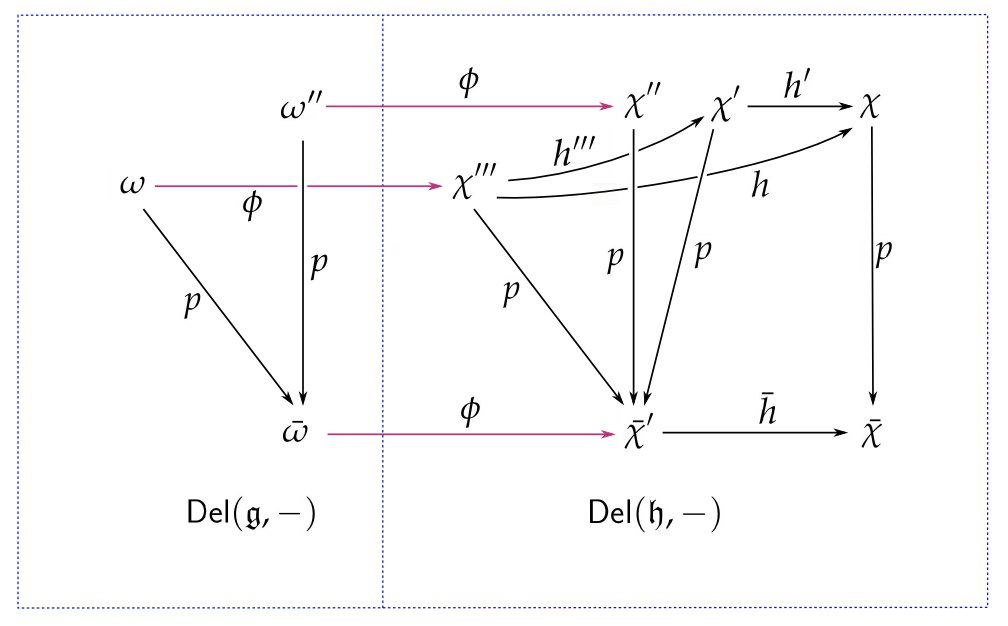}
\caption{Illustration for the proof of Lemma \ref{lem:1}.
The diagram is commutative.} 
\label{fig:1}
\end{figure}

\begin{lem} \label{lem:5}
Let $\om \in \opn{MC}(\g, R)$ and 
$\chi := \phi(\om) \in \opn{MC}(\h, R)$.
\begin{enumerate}
\item The homomorphism of DG Lie algebras 
\[ \phi : (\m \otimes \g)_{\om} \to (\m \otimes \h)_{\chi} \]
is a quasi-isomorphism.

\item The group homomorphism 
\[ \phi : \opn{G}^{\mrm{r}}(\g, R)(\om, \om) \to 
\opn{G}^{\mrm{r}}(\h, R)(\chi, \chi) \]
is an isomorphism.
\end{enumerate}
\end{lem}

\begin{proof}
(1) This is done by induction on $l(R)$. If $l(R) = 1$ then 
$\m^2 = 0$, so $\d_{\om} = \d$ and $\d_{\chi} = \d$. 
Since $\phi : \g \to \h$ is a 
quasi-isomorphisms, and since $\m$ is flat over $\K$, the assertion is true.

Now assume that $l(R) \geq 2$. Since $\m \n = 0$ it follows that 
$\d_{\om}|_{\n \otimes \g} = \d|_{\n \otimes \g}$; 
and likewise for $\n \otimes \h$. 
Let $\bar{\om} := p(\om)$ and $\bar{\chi} := p(\chi)$. 
We get a commutative diagram of complexes of $R$-modules 
\[ \UseTips \xymatrix @C=6ex @R=6ex {
0
\ar[r]
&
\n \otimes \g
\ar[r]
\ar[d]^{\phi}
&
(\m \otimes \g)_{\om}
\ar[r]^{p}
\ar[d]^{\phi}
&
(\bar{\m} \otimes \g)_{\bar{\om}}
\ar[r]
\ar[d]^{\phi}
& 
0
\\
0
\ar[r]
&
\n \otimes \h
\ar[r]
&
(\m \otimes \h)_{\chi}
\ar[r]^{p}
&
(\bar{\m} \otimes \h)_{\bar{\chi}}
\ar[r]
&
0
} \]
with exact rows. By induction the right vertical arrow is a quasi-isomorphism;
and the left vertical arrow is a quasi-isomorphism by the same argument given
in the case $l(R) = 1$.
Therefore the middle vertical arrow is a quasi-isomorphism.

\medskip \noindent
(2) Combine item (1) above and Proposition \ref{prop:2}.
\end{proof}

\begin{lem} \label{lem:2}
Let $\om, \om' \in \opn{MC}(\g, R)$, and define 
$\bar{\om} := p(\om)$, 
$\bar{\om}' := p(\om')$, 
$\chi := \phi(\om)$,
$\chi' := \phi(\om')$, 
$\bar{\chi} := \phi(\bar{\om})$ and 
$\bar{\chi}' := \phi(\bar{\om}')$.
Let 
\[ \bar{g} \in \opn{G}^{\mrm{r}}(\g, \bar{R})(\bar{\om}, \bar{\om}') , \]
\[ \bar{h} := \phi(\bar{g}) \in \opn{G}^{\mrm{r}}(\h, \bar{R})
(\bar{\chi}, \bar{\chi}') , \]
and
\[ h \in \opn{G}^{\mrm{r}}(\h, R)(\chi, \chi')  / \bar{h} . \]
Then there exists a unique element
\[ g \in \opn{G}^{\mrm{r}}(\g, R)(\om, \om') / \bar{g} \]
such that $\phi(g) = h$.
\end{lem}

\begin{proof}
The proof is very similar to the proof of ``Full''
in \cite[Subsection 2.11]{GM}. (Note however that there is a mistake in loc.\
cit. In our notation, what is done there is referring to the obstruction class 
$o_1(\om, \om')$, but this is not defined since 
$p(\om) \neq p(\om')$ in general.) The proof is illustrated in Figure
\ref{fig:2}.

Choose an arbitrary lift $g'' \in \opn{G}(\g, R)$ of $\bar{g}$, 
namely $\bar{g} = p(\eta_1(g''))$.
Define 
\[ \om'' := \opn{Af}(g'')(\om) \in \opn{MC}(\g, R) , \]
\[ h'' := \phi(g'') \in \opn{G}(\h, R) , \]
\[ \chi'' := \phi(\om'') \in \opn{MC}(\h, R) \]
and
\[ h' := h \cdot (h'')^{-1} \in 
\opn{G}(\h, R)(\chi'', \chi') / 1 . \]
Since $\om'', \om' \in \opn{MC}(\g, R) / \bar{\om}'$
the obstruction class 
$o_1(\om'', \om')$
is defined, and it satisfies 
\[ \mrm{H}^1 (\phi)(o_1(\om'', \om')) = 
o_1(\chi'', \chi') = 0 \]
because $h'$ exists. By assumption the homomorphism
$\mrm{H}^1 (\phi)$ is injective, and we conclude that 
$o_1(\om'', \om') = 0$. So there exists some 
$g' \in \opn{G}(\g, R)(\om'', \om') / 1$.
Let 
\[ g''' := g' \cdot g'' \in \opn{G}(\g, R)(\om, \om') . \]
Then $p(\eta_1(g''')) = \bar{g}$, and hence
\[ \eta_1(g''') \in \opn{G}^{\mrm{r}}(\g, R)(\om, \om') / \bar{g} . \]

By Lemma \ref{lem:5}(2) we have a group isomorphism 
\[ \phi : \opn{G}^{\mrm{r}}(\g, R)(\om, \om) / 1 \to 
\opn{G}^{\mrm{r}}(\h, R)(\chi, \chi) / 1 . \]
Since the set 
$\opn{G}^{\mrm{r}}(\g, R)(\om, \om') / \bar{g}$
is nonempty (it contains $\eta_1(g''')$), it admits a simply transitive action
by the group 
$\opn{G}^{\mrm{r}}(\g, R)(\om, \om) / 1$. 
Therefore the function
\[ \phi : \opn{G}^{\mrm{r}}(\g, R)(\om, \om') / \bar{g} \to 
\opn{G}^{\mrm{r}}(\h, R)(\chi, \chi') / \bar{h} \]
is bijective. We see that there is a unique element 
$g \in \opn{G}^{\mrm{r}}(\g, R)(\om, \om') / \bar{g}$
such that 
$\phi(g) = h$.
\end{proof}

\begin{figure}
\includegraphics[scale=0.35]{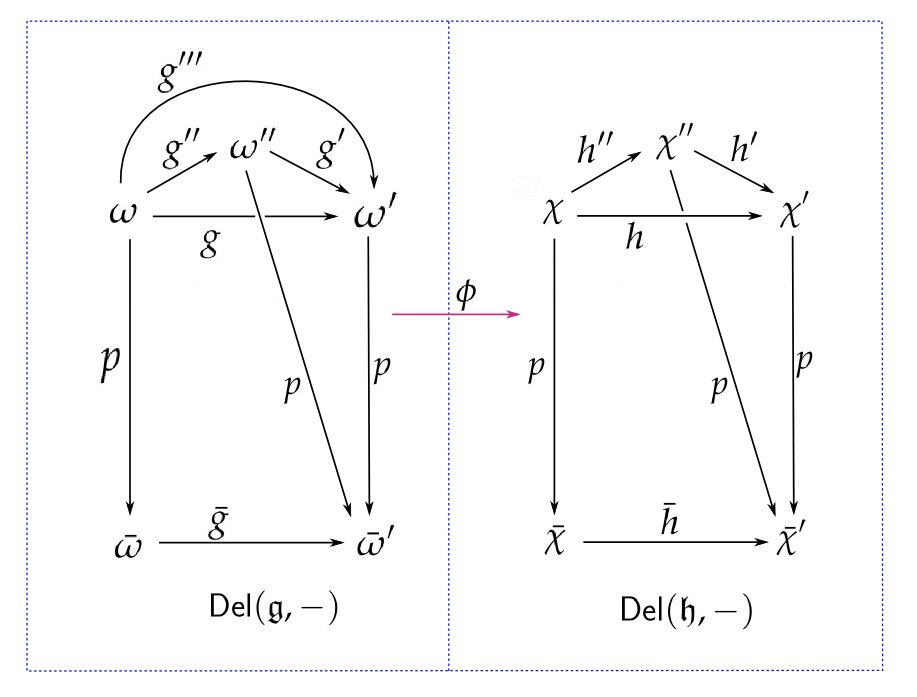}
\caption{Illustration for the proof of Lemma \ref{lem:2}.
Some of the arrows, like $g$ and $h$, belong to the groupoid
$\mbf{Del}^{\mrm{r}}(-, -)$. 
Other arrows, like $g''$ and $h''$, belong to the groupoid
$\mbf{Del}(-, -)$. 
The function $\phi$  sends $\om \mapsto \chi$,
$\bar{\om} \mapsto \bar{\chi}$, $g \mapsto h$, etc. 
The whole diagram is commutative.} 
\label{fig:2}
\end{figure}

\section{DG Lie Quasi-isomorphisms -- Pronilpotent Algebras}
\label{sec:dg-quasi-comp}

In this section we extend \cite[Theorem 2.4]{GM} (attributed to Deligne) to the
case of complete noetherian local rings and unbounded DG Lie algebras. 
This is Theorem \ref{thm:2}.

Let $(R, \m)$ be a complete parameter algebra, and let
$\phi : \g \to \h$ be a DG Lie algebra quasi-isomorphism. 
As in Section \ref{sec:facts}, for any $j \in \N$ we write 
$R_j := R / \m^{j+1}$ and $\m_j := \m / \m^{j+1}$. We denote by 
$p_j : R \to R_j$ and $p_{j, i} : R_j \to R_i$ the canonical projections 
(for $j \geq i$).
Let $\n_j := \m^{j} / \m^{j + 1}$, which is an ideal in $R_j$ satisfying \lb
$\m_j \n_j = 0$ and 
\[ \n_j = \opn{Ker}(p_{j, j-1} : R_j \to R_{j - 1}) . \]
Thus $R_{j - 1} \cong R_j / \n_j$.

\begin{lem} \label{lem:6}
\begin{enumerate}
\item Let $\om \in  \opn{MC}(\g, R)$ and
$\chi := \phi(\om) \in \opn{MC}(\h, R)$.
Then the homomorphism of DG Lie algebras
\[ \phi :  (\m \hatotimes \g)_{\om} \to 
(\m \hatotimes \h)_{\chi} \]
is a quasi-isomorphism. 

\item The canonical function 
\[ \opn{MC}(\g, R) \to \lim_{\leftarrow j}\, \opn{MC}(\g, R_j) \]
is bijective.

\item For any $\om, \om' \in \opn{MC}(\g, R)$ the canonical function 
\[ \opn{G}^{\mrm{r}}(\g, R)(\om, \om') \to \lim_{\leftarrow j}\, 
\opn{G}^{\mrm{r}} ( \g, R_j) \bigl( p_j(\om), p_j(\om') \bigr) \]
is surjective.
\end{enumerate}
\end{lem}

Of course items (2-3) refer also to $\h$.

\begin{proof}
(1) We forget the Lie brackets. Let $M$ be the mapping cone of the homomorphism
of complexes of $R$-modules
\[ \phi :  (\m \hatotimes \g)_{\om} \to (\m \hatotimes \h)_{\chi} . \]
So 
\[ M = (\m \hatotimes \g)_{\om}[1] \oplus (\m \hatotimes \h)_{\chi} , \]
with a suitable differential. 
For any $j \geq 0$ let $\om_j := p_j(\om)$ and $\chi_j := p_j(\chi)$.
We have an inverse system of homomorphisms of complexes
\[ \phi_j :  (\m_j \hatotimes \g)_{\om_j} \to 
(\m_j \hatotimes \h)_{\chi_j} , \]
and we denote by $M_j$ the mapping cone of $\phi_j$.  
Then each $p_j : M \to M_j$ is surjective, and 
$M \cong \lim_{\leftarrow j}\, M_j$.
According to Lemma \ref{lem:5}(1) the complexes $M_j$ are acyclic. 
Therefore, using the Mittag-Leffler argument, the complex $M$ is also acyclic.

\medskip \noindent
(2) This is because $\m \hatotimes \g^1$ is $\m$-adically complete, and 
$\opn{MC}(\g, R)$ is a closed subset in it (w.r.t.\ the $\m$-adic metric).

\medskip \noindent
(3) Write $\om'_j := p_j(\om')$. 
Suppose we are given a sequence $\{ g_j \}_{j \in \N}$ of elements
\[ g_j \in \opn{G}^{\mrm{r}} ( \g, R_j)(\om_j, \om'_j ) \]
such that $p_{j, j-1}(g_j) = g_{j-1}$.
We are going to find a sequence $\{ \til{g}_j \}_{j \in \N}$ of elements
\[ \til{g}_j \in \opn{G} ( \g, R_j)(\om_j, \om'_j ) \]
such that $p_{j, j-1}(\til{g}_j) = \til{g}_{j-1}$
and $\eta_1(\til{g}_j) = g_j$ for all $j$.
Since the group $\opn{G}(\g, R)$ is 
complete w.r.t.\ its $\m$-adic filtration, the limit
\[ \til{g} := \lim_{\leftarrow j}\, \til{g}_j \in \opn{G}(\g, R) \]
exists; and by continuity 
$\til{g} \in \opn{G}(\g, R)(\om, \om')$.
Then the element
\[ g := \eta_1(\til{g}) \in \opn{G}^{\mrm{r}}(\g, R)(\om, \om') \]
satisfies $p_j(g) = g_j$ for all $j$.

Here is the recursive construction of the sequence 
$\{ \til{g}_j \}_{j \in \N}$.
For $j = 0$ we take $\til{g}_0 := 1$. Now assume that $j \geq 1$ and we have a 
sequence $(\til{g}_0, \ldots, \til{g}_{j-1})$ as required.
Choose any element 
$\til{g}'_j \in \opn{G}(\g, R_j)(\om_j, \om'_j)$
such that $\eta_1(\til{g}'_j) = g_j$. 
Then 
\[ \eta_1(p_{j, j-1}(\til{g}'_j)) = p_{j, j-1}(\eta_1(\til{g}'_j)) = 
p_{j, j-1}(g_j) = g_{j-1} = \eta_1(\til{g}_{j-1}) . \]
There is some $\bar{a} \in N^{\mrm{r}}(\g, R_{j-1})_{\om_{j-1}}$
such that 
$\bar{a} \cdot p_{j, j-1}(\til{g}'_j) = \til{g}_{j-1}$.
Choose any $a \in N^{\mrm{r}}(\g, R_j)_{\om_j}$
lifting $\bar{a}$. Then 
$\til{g}_j := a \cdot \til{g}'_j$ will satisfy 
$p_{j, j-1}(\til{g}_j) = \til{g}_{j-1}$ and
$\eta_1(\til{g}_j) = g_j$.
\end{proof}

Here is the main result of this section. We denote the identity automorphism of
$R$ by $\bsym{1}_R$.

\begin{thm} \label{thm:2}
Let $(R, \m)$ be a parameter algebra over $\K$, and let $\phi : \g \to \h$ be a
DG Lie algebra quasi-isomorphism over $\K$. Then the function 
\[ \ol{\opn{MC}}(\bsym{1}_R \otimes \phi) : \ol{\opn{MC}}(\m \hatotimes \g) \to
\ol{\opn{MC}}(\m \hatotimes \h)  \]
is bijective. Moreover, the morphism of groupoids
\[ \mbf{Del}^{\mrm{r}}(\phi, R) : \mbf{Del}^{\mrm{r}}(\g, R) 
\to \mbf{Del}^{\mrm{r}}(\h, R)  \]
is an equivalence. 
\end{thm}

Observe that there are no finiteness nor boundedness conditions on $\g$ and
$\h$. 

\begin{proof}
We will prove these assertions:
\begin{enumerate}
\rmitem{a} The function $\ol{\opn{MC}}(\bsym{1}_R \otimes \phi)$
is surjective.

\rmitem{b}  The function $\ol{\opn{MC}}(\bsym{1}_R \otimes \phi)$
is injective.

\rmitem{c} For any 
$\om \in \opn{MC}(\g, R)$ the group homomorphism 
\[ \opn{G}^{\mrm{r}}(\phi, R) : \opn{G}^{\mrm{r}}(\g, R)(\om, \om) \to 
\opn{G}^{\mrm{r}}(\h, R) \bigl( \phi(\om), \phi(\om) \bigr) \]
is bijective. 
\end{enumerate}

Assertions (a-b) say that the function
$\ol{\opn{MC}}(\bsym{1}_R \otimes \phi)$
is bijective. Then assertion (c) implies that the function 
\[ \phi :  \opn{G}^{\mrm{r}}(\g, R)(\om, \om') \to 
\opn{G}^{\mrm{r}}(\h, R)(\phi(\om), \phi(\om')) \]
is bijective for every $\om, \om' \in \opn{MC}(\g, R)$.
Hence $\mbf{Del}^{\mrm{r}}(\phi, R)$ is an equivalence.

\medskip \noindent
Proof of (a). Here it is more convenient to work with the groupoids
$\mbf{Del}(-, -)$. 
Suppose we are given $\chi \in \opn{MC}(\h, R)$. We will find 
elements $\om \in \opn{MC}(\g, R)$ and 
$h \in \opn{G}(\h, R)(\phi(\om), \chi)$. 

Define $\chi_j := p_j(\chi) \in \opn{MC}(\h, R_j)$.
We are going to find a sequence $\{ \om_j \}_{j \in \N}$ of elements
$\om_j \in \opn{MC}(\g, R_j)$, and a sequence $\{ h_j \}_{j \in \N}$ of
elements
\[ h_j \in  \opn{G}(\h, R_j)(\phi(\om_j), \chi_j) , \]
such that $p_{j, j-1}(\om_j) = \om_{j-1}$ and 
$p_{j, j-1}(h_j) = h_{j-1}$ for all $j$.
This is done by induction on $j$.
For $j = 0$ we take $\om_0 := 0$ and $h_0 := 1 = \exp(0)$. 
Now consider $j \geq 1$. Assume that we have found sequences 
$(\om_0, \ldots, \om_{j-1})$ and 
$(h_0, \ldots, h_{j-1})$ satisfying the required conditions. 
By Lemma \ref{lem:1}, applied to the artinian local ring $R_j$,
and the elements $\chi_j, \om_{j-1}$ and $h_{j-1}$, 
there exist elements $\om_j$ and $h_j$ as required. 

Now the $R$-module $\m \hatotimes \g^1$ is $\m$-adically complete, 
and the set $\opn{MC}(\g, R)$ is closed inside 
$\m \hatotimes \g^1$ (with respect to the $\m$-adic metric). 
Hence the limit 
$\om := \lim_{\leftarrow j}\, \om_j$
belongs to $\opn{MC}(\g, R)$. 
The gauge group $\opn{G}(\h, R)$ is complete, since it is pronilpotent 
(with respect to its $\m$-adic filtration).
We get an element 
\[ h := \lim_{\leftarrow j}\, h_j \in \opn{G}(\h, R) . \]
By continuity we see that 
$\opn{Af}(h)(\phi(\om)) = \chi$.

\medskip \noindent
Proof of (b). Here we prefer to work with the groupoids
$\mbf{Del}^{\mrm{r}}(-,-)$.
Take $\om, \om' \in \opn{MC}(\g, R)$, and define 
$\chi := \phi(\om)$ and $\chi' := \phi(\om')$. 
Suppose we are given 
$h \in \opn{G}^{\mrm{r}}(\h, R)(\chi, \chi')$;
we will find 
$g \in \opn{G}^{\mrm{r}}(\g, R)(\om, \om')$. 
(We do not care whether $\phi(g) = h$ or not.)

Define 
$\om_j := p_j(\om)$,
$\om'_j := p_j(\om')$,
$\chi_j := p_j(\chi)$,
$\chi'_j := p_j(\chi')$
and 
$h_j := p_j(h)$. 
We will find a sequence $\{ g_j \}_{j \in \N}$ of elements 
$g_j \in  \opn{G}^{\mrm{r}}(\g, R_j)(\om_j, \om'_j)$
such that 
$\phi(g_j) = h_j$
and 
$p_{j, j-1}(g_j) = g_{j-1}$.
Then, by Lemma \ref{lem:6}(3) there is an element 
$g \in \opn{G}^{\mrm{r}}(\g, R)(\om, \om')$
such that $p_j(g) = g_j$ for every $j$. 

We construct the sequence $\{ g_j \}_{j \in \N}$ by induction on $j$. For
$j = 0$ we take $g_0 := 1$. 
Now let $j \geq 1$, and suppose that we have a sequence 
$(g_0, \ldots, g_{j-1})$ as required. According to Lemma \ref{lem:2}, applied
to the artinian local ring $R_j$, there exists an element 
$g_j \in \opn{G}^{\mrm{r}}(\g, R_j)(\om_j, \om'_j)$
such that $p_{j, j-1}(g_j) = g_{j-1}$ and $\phi(g_j) = h_j$.

\medskip \noindent
Proof of (c). This is the complete version of the proof of Lemma \ref{lem:5}(2).
By Lemma \ref{lem:6}(1) the function 
$\mrm{H}^0(\phi_{R, \om})$ is bijective. The claim now follows from Proposition 
\ref{prop:2}.
\end{proof}

\begin{rem}
Assume $\g$ is abelian (i.e.\ the Lie bracket is zero), so that 
$\m \hot \g$ is just a complex of $R$-modules, and 
\[ \ol{\opn{MC}}(\g, R) = \mrm{H}^1 (\m \hot \g) . \]
In this case the Deligne groupoid $\mbf{Del}(\g, R)$ is the truncation 
\[ \cdots \to 0 \to \m \hot \g^0 \to 
\opn{Ker} \bigl( \d : \m \hot \g^{1} \to \m \hot \g^2 \bigr)
\to 0 \to \cdots , \]
whereas as the reduced Deligne groupoid $\mbf{Del}^{\mrm{r}}(\g, R)$
is the truncation
\[ \cdots \to 0 \to 
\opn{Im} \bigl( \d : \m \hot \g^{-1} \to \m \hot \g^0 \bigr)
\to 
\opn{Ker} \bigl( \d : \m \hot \g^{1} \to \m \hot \g^2 \bigr)
\to 0 \to \cdots , \]
both concentrated in the degree range $[0, 1]$.
Let $\h$ be another abelian DG Lie algebra. 
It is clear that a quasi-isomorphism of complexes $\phi : \g \to \h$ will
induce a quasi-isomorphism 
\[ \mbf{Del}^{\mrm{r}}(\phi, R) : 
\mbf{Del}^{\mrm{r}}(\g, R) \to \mbf{Del}^{\mrm{r}}(\h, R) . \]
This is a special case of Theorem \ref{thm:2}.
\end{rem}

\begin{rem} \label{rem:11}
Presumably Theorem \ref{thm:2} can be extended to the following more general
situation: $\g$ and $\h$ are $R$-linear DG Lie algebras, such that all the
$R$-modules $\g^i$ and $\h^i$ are $\m$-adically complete, and the graded Lie
algebras $\opn{gr}_{\m} (\g)$ and $\opn{gr}_{\m} (\h)$ are abelian. 
We are given an $R$-linear DG Lie algebra homomorphism 
$\phi : \g \to \h$ such that 
\[ \opn{gr}_{\m}(\phi) : \opn{gr}_{\m} (\g) \to \opn{gr}_{\m} (\h) \]
is a quasi-isomorphism. Then 
\[ \ol{\opn{MC}}(\phi) : \ol{\opn{MC}}(\g) \to \ol{\opn{MC}}(\h) \]
is bijective. Cf.\ \cite[Theorem 2.1]{Ge} for the corresponding nilpotent case.
\end{rem}

\section{Some Facts on $2$-Groupoids}
\label{sec:2-grpd}

Let us recall that a (strict) {\em $2$-groupoid} $\bsym{G}$ is a groupoid
enriched in the monoidal category of groupoids. Another way of saying this is
that a $2$-groupoid $\bsym{G}$ is a $2$-category in which all $1$-morphisms and
$2$-morphisms are invertible. A comprehensive review of $2$-categories and
related constructions is available in \cite[Section 1]{Ye3}. 
See also \cite{Ma, Bw, Ge}.

We wish to make things as explicit as possible, to make calculations
(both in Section \ref{sec:del-2-grpd} of this paper, and in the new
version of \cite{Ye4}) easier.
A (small strict) $2$-groupoid $\bsym{G}$ is made up of the following
ingredients: there
is a set $\opn{Ob}(\bsym{G})$, whose elements are the {\em objects}
of $\bsym{G}$. For any $x, y \in \opn{Ob}(\bsym{G})$ there is a set 
$\bsym{G}(x, y)$, whose elements are called the {\em $1$-morphisms} from $x$
to $y$. Given $f \in \bsym{G}(x, y)$, we write $f : x \to y$. 
For any $f, g \in \bsym{G}(x, y)$ there is a set 
$\bsym{G}(x, y)(f, g)$, whose elements are called the {\em $2$-morphisms} from
$f$ to $g$. For $a \in \bsym{G}(x, y)(f, g)$ we write 
$a : f \twoto g$. 

There are three types of composition operations in $\bsym{G}$. There is {\em
horizontal composition of $1$-morphisms}: given $f_1 : x_0 \to x_1$ and 
$f_2 : x_1 \to x_2$, their composition is
$f_2 \circ f_1 : x_0 \to x_2$. Suppose we are also given $1$-morphisms 
$g_i : x_{i-1} \to x_i$ and $2$-morphisms $a_i : f_i \twoto g_i$. Then there is
a $2$-morphism 
\[ a_2 \circ a_1 : f_2 \circ f_1 \twoto g_2 \circ g_1 . \]
This is {\em horizontal composition of $2$-morphisms}. 
If we are also given \lb $1$-morphisms 
$h_i : x_{i-1} \to x_i$ and $2$-morphisms $b_i : g_i \twoto h_i$,
then there is the {\em vertical composition} (of $2$-morphisms) 
$b_i \ast a_i : f_i \twoto h_i$. 
There are pretty diagrams to display all of this (see \cite[Section 1]{Ye3} or
many other references). 

For every $x \in \opn{Ob}(\bsym{G})$ there is the identity $1$-morphism 
$\bsym{1}_x : x \to x$, and for every 
$f \in \bsym{G}(x, y)$ there is the identity $2$-morphism 
$\bsym{1}_f : f \twoto f$. 

Here are the conditions required for the structure $\bsym{G}$ to be a 
$2$-groupoid: 
\begin{itemize}
\item The set $\opn{Ob}(\bsym{G})$, together with the $1$-morphisms
$f : x \to y$, horizontal composition $g \circ f$, and the identity morphisms 
$\bsym{1}_x$, is a groupoid. 
We refer to this groupoid as the $1$-truncation of $\bsym{G}$. 

\item For every $x, y \in \opn{Ob}(\bsym{G})$, the set 
$\bsym{G}(x,y)$, together with the $2$-morphisms 
$a : f \twoto g$, vertical composition
$b \ast a$, and the identity morphisms $\bsym{1}_f$, is a groupoid. We refer
to it as the vertical groupoid above $(x, y)$. 

\item Horizontal composition of $2$-morphisms is associative, 
$\bsym{1}_{g \circ f} = \bsym{1}_g \circ \bsym{1}_f$ whenever $g$ and $f$ are
composable, and the $2$-morphisms $\bsym{1}_{\bsym{1}_x}$ are identities for
horizontal composition.

\item The exchange condition: given $f_i, g_i, h_i : x_{i-1} \to x_i$,
$a_i : f_i \twoto g_i$ and $b_i : g_i \twoto h_i$, one has
\[ (b_2 \ast a_2) \circ (b_1 \ast a_1) = 
(b_2 \circ b_1) \ast (a_2 \circ a_1) , \]
as $2$-morphisms $f_2 \circ f_1 \twoto h_2 \circ h_1$.
\end{itemize}

A consequence of these four conditions is that $2$-morphisms are invertible for
horizontal composition. Indeed, given $a : f \twoto g$ in 
$\bsym{G}(x, y)$,  its horizontal inverse 
$a^{- \circ} : f^{-1} \twoto g^{-1}$ is given by the formula
$a^{- \circ} = \bsym{1}_{g^{-1}} \circ a^{- \ast} \circ \bsym{1}_{f^{-1}}$,
where 
$a^{- \ast} : g \twoto f$ is the vertical inverse of $a$. 

Suppose $\bsym{H}$ is another $2$-groupoid. A (strict) {\em $2$-functor} 
$F : \bsym{G} \to \bsym{H}$ is a collection of functions 

\[ \begin{aligned}
F & : \opn{Ob}(\bsym{G}) \to \opn{Ob}(\bsym{H})
\\
F & : \bsym{G}(x_0, x_1) \to \bsym{H} \bigl( F(x_0), F(x_1) \bigr)
\\
F & : \bsym{G}(x_0, x_1)(g_0, g_1) \to 
\bsym{H} \bigl( F(x_0), F(x_1) \bigr) \bigl( F(g_0), F(g_1) \bigr) 
\end{aligned} \]
that respect the various compositions and identity morphisms. We denote by
$\cat{2-Grpd}$ the category consisting of $2$-groupoids and $2$-functors between
them.  

Consider a $2$-groupoid $\bsym{G}$. There is an equivalence relation on the set 
$\opn{Ob}(\bsym{G})$, given by existence of $1$-morphisms, i.e.\ $x \sim y$ if 
$\bsym{G}(x, y) \neq \emptyset$. We let 
$\pi_0(\bsym{G}) := \opn{Ob}(\bsym{G}) / {\sim}$. 

For objects $x, y \in \opn{Ob}(\bsym{G})$ there is an equivalence relation on
the set $\bsym{G}(x, y)$, given by existence of $2$-morphisms:
$f \sim g$ if $\bsym{G}(x, y)(f, g) \neq \emptyset$. We let 
\[ \pi_1(\bsym{G}, x, y) := \bsym{G}(x, y) / {\sim} . \]
We define $\pi_1(\bsym{G})$ to be the  groupoid with object set
$\opn{Ob}(\bsym{G})$, \lb
morphism sets $\pi_1(\bsym{G}, x, y)$, and composition induced by horizontal
composition in $\bsym{G}$. Thus $\pi_1(\bsym{G})$ is a quotient groupoid of the
$1$-truncation of $\bsym{G}$, and \lb
$\pi_0(\pi_1(\bsym{G})) = \pi_0(\bsym{G})$.
We write $\pi_1(\bsym{G}, x) := \pi_1(\bsym{G}, x, x)$, which is a group.
We also define 
\[ \pi_2(\bsym{G}, x)  := \bsym{G}(x, x)(\bsym{1}_x, \bsym{1}_x) . \]
This is an abelian group. 

The homotopy set $\pi_0(\bsym{G})$ and groups $\pi_i(\bsym{G}, x)$ are
functorial in an obvious way. 
A morphism $F : \bsym{G} \to \bsym{H}$ in 
$\cat{2-Grpd}$ is called a {\em weak equivalence} if the functions 
\[ \begin{aligned}
\pi_0(F) & : \pi_0(\bsym{G}) \to \pi_0(\bsym{H})
\\
\pi_1(F, x) & : \pi_1(\bsym{G}, x) \to \pi_1(\bsym{H}, F(x))
\\
\pi_2(F, x) & : \pi_2(\bsym{G}, x) \to \pi_2(\bsym{H}, F(x))
\end{aligned} \]
are bijective for all $x \in \opn{Ob}(\bsym{G})$. 

It will be useful to relate the concept of $2$-groupoid to the less
familiar concept of {\em crossed module over a groupoid}, recalled below.

For a groupoid $\bsym{G}$ and an object $x \in \opn{Ob}(\bsym{G})$
we denote by $\bsym{G}(x) := \bsym{G}(x, x)$, the automorphism group of $x$. 
The composition in $\bsym{G}$ is $\circ = \circ_{\bsym{G}}$.

Let $\bsym{G}$ and $\bsym{N}$ be groupoids, such that 
$\opn{Ob}(\bsym{G}) = \opn{Ob}(\bsym{N})$. An {\em action} $\Psi$ of $\bsym{G}$
on 
$\bsym{N}$ is a collection of group isomorphisms 
\[ \Psi(g) : \bsym{N}(x) \iso \bsym{N}(y) \]
for all $x, y \in \opn{Ob}(\bsym{G})$ and $g \in \bsym{G}(x, y)$, 
such that 
\[ \Psi(h \circ g) = \Psi(h) \circ \Psi(g) \]
whenever $g$ and $h$ are composable, and 
$\Psi(\bsym{1}_x) = \bsym{1}_{\bsym{N}(x)}$.

\begin{exa}
Let $\bsym{G}$ be any groupoid. The adjoint action $\opn{Ad}_{\bsym{G}}$ of
$\bsym{G}$ on itself is defined by
\[ \opn{Ad}_{\bsym{G}}(g)(h) := g \circ h \circ g^{-1} \]
for $g \in \bsym{G}(x, y)$ and $h \in \bsym{G}(x)$.
\end{exa}

A {\em crossed module over a groupoid}, or a {\em crossed groupoid} 
for short, is data \lb
$(\bsym{G}, \bsym{N}, \Psi, \opn{D})$ consisting of:
\begin{itemize}
\item Groupoids $\bsym{G}$ and $\bsym{N}$, such that
$\bsym{N}$ is totally disconnected, and  
$\opn{Ob}(\bsym{N}) = \opn{Ob}(\bsym{G})$.

\item An action $\Psi$ of $\bsym{G}$ on $\bsym{N}$, called the {\em twisting}.

\item A morphism of groupoids (i.e.\ a functor)
$\opn{D} : \bsym{N} \to \bsym{G}$, called the {\em feedback}, which is the 
identity on objects.
\end{itemize}

These are the conditions:
\begin{enumerate}
\rmitem{i} The morphism $\opn{D}$ is $\bsym{G}$-equivariant with respect to
the actions $\Psi$ and $\opn{Ad}_{\bsym{G}}$. Namely 
\[ \opn{D}(\Psi(g)(a)) = \opn{Ad}_{\bsym{G}}(g)(\opn{D}(a)) \]
in the group $\bsym{G}(y)$, for any $x, y \in \opn{Ob}(\bsym{G})$, 
$g \in \bsym{G}(x, y)$ and $a \in \bsym{N}(x)$.

\rmitem{ii} For any $x \in \opn{Ob}(\bsym{G})$ and 
$a \in \bsym{N}(x)$ there is equality
\[ \Psi(\opn{D}(a)) = \opn{Ad}_{\bsym{N}(x)}(a) , \]
as automorphisms of the group $\bsym{N}(x)$.
\end{enumerate}

\begin{exa}
If $\bsym{G}$ and $\bsym{N}$ are groups, namely 
$\opn{Ob}(\bsym{G}) = \opn{Ob}(\bsym{N}) = \{ 0 \}$, then 
a crossed groupoid is just a crossed module.
\end{exa}

\begin{exa} \label{exa:100}
Let $G$ be a group acting on a set $X$. For $x \in X$ let $G(x)$ denote the
stabilizer of $x$. Let $\{ N_x \}_{x \in X}$ be a 
collection of groups. Assume that for every $g \in G$ and $x \in X$ there is
given a group isomorphism $\Psi(g) : N_x \iso N_{g(x)}$, and these satisfy the
functoriality conditions of an action. 
Also assume there are
group homomorphisms $\opn{D}_x : N_x \to G(x)$ such that 
\[ \opn{Ad}_{G}(g) \circ \opn{D}_x = \opn{D}_{g(x)} \circ \, \Psi(g) \]
for any $g \in G$ and $x \in X$. 

Define $\bsym{G}$ to be the transformation groupoid associated to the
action of $G$ on $X$. And define $\bsym{N}$ to be the totally
disconnected groupoid with $\opn{Ob}(\bsym{N}) := X$ and 
$\bsym{N}(x) := N_x$. Then 
$( \bsym{G}, \bsym{N}, \Psi, \opn{D})$ is a crossed
groupoid, which we call the {\em transformation crossed groupoid} associated to
the action of $G$ on $\{ N_x \}_{x \in X}$.
\end{exa}

\begin{exa}
Suppose $\bsym{G}$ is any groupoid, and $\bsym{N}$ is a {\em normal subgroupoid}
of $\bsym{G}$, in the sense of \cite{Bw, Ye4}. Let 
$\opn{D} : \bsym{N}\to \bsym{G}$ be the inclusion, and let 
$\Psi$ be the restriction of $\opn{Ad}_{G}$ to $\bsym{N}$. Then 
$(\bsym{G}, \bsym{N}, \Psi, \opn{D})$ is a crossed groupoid.
\end{exa}

It is known that crossed groupoids are the same as $2$-groupoids
(cf.\ \cite{Bw}). We will now give a precise statement of this fact. 

\begin{prop} \label{prop:102}
Let $(\bsym{H}, \bsym{N}, \Psi, \opn{D})$ 
be a crossed groupoid. Then there is a unique $2$-groupoid 
$\bsym{G}$ with these properties:
\begin{enumerate}
\rmitem{i} The $1$-truncation of $\bsym{G}$ is the same groupoid as
$\bsym{H}$. Namely $\opn{Ob}(\bsym{G}) = \opn{Ob}(\bsym{H})$,
$\bsym{G}(x, y) = \bsym{H}(x, y)$, the identity morphisms $\bsym{1}_x$ are the
same, and the horizontal composition is 
$g \circ_{\bsym{G}} f = g \circ_{\bsym{H}} f$
for $f : x \to y$ and $g : y \to z$.

\rmitem{ii} For any $f, g : x \to y$ in $\bsym{G}$ we have
\[ \bsym{G}(x, y)(f, g) = 
\{ a \in \bsym{N}(x) \mid g = f \circ_{\bsym{H}} \opn{D}(a) \} . \]
The identity morphism $\bsym{1}_f \in \bsym{G}(x, y)(f, f)$ is 
$\bsym{1}_x \in \bsym{N}(x)$.
Given $h : x \to y$, $a : f \twoto g$ and $b : g \twoto h$,
the vertical composition is 
$b \ast_{\bsym{G}} a = a \circ_{\bsym{N}} b$.

\rmitem{iii} For any $x_0, x_1, x_2 \in \opn{Ob}(\bsym{G})$, any
$f_i, g_i : x_{i-1} \to x_i$ and any
$a_i : f_i \twoto g_i$ in $\bsym{G}$, 
the horizontal composition $a_2 \circ_{\bsym{G}} a_1$ satisfies
\[ a_2 \circ_{\bsym{G}} a_1 =  
\Psi(f_1^{-1})(a_2) \circ_{\bsym{N}} a_1 . \]
\end{enumerate}

Moreover, any $2$-groupoid $\bsym{G}$ arises this way. 
\end{prop}

\begin{proof}
It is easy to verify that the conditions of a $2$-groupoid hold. 

Conversely, suppose $\bsym{G}$ is any $2$-groupoid. Define the groupoid
$\bsym{G}_1$ to be the $1$-truncation of $\bsym{G}$. 

For any $x \in \opn{Ob}(\bsym{G})$ consider the set of $2$-morphisms
\begin{equation} \label{eqn:105}
\bsym{G}_2(x) := \coprod_{g \in \bsym{G}(x, x)} 
\bsym{G}(x, x)(\bsym{1}_x, g) .
\end{equation}
This is a group under horizontal composition $\circ_{\bsym{G}}$
of $2$-morphisms, with identity element 
$\bsym{1}_{\bsym{1}_x}$. There is a group homomorphism 
$\opn{D} : \bsym{G}_2(x) \to \bsym{G}_1(x)$, defined by 
\begin{equation} \label{eqn:104}
\opn{D}(a : \bsym{1}_x \twoto g) := g  .
\end{equation}
Let $\bsym{G}_2$ be the totally disconnected groupoid with set of objects 
$\opn{Ob}(\bsym{G})$, and automorphism groups $\bsym{G}_2(x)$ as defined above.
Then $\opn{D} : \bsym{G}_2 \to \bsym{G}_1$
is a morphism of groupoids.

A $1$-morphism $f : x \to y$ in $\bsym{G}$ induces a group isomorphism 
\[ \opn{Ad}_{\bsym{G}_1 \curvearrowright \bsym{G}_{2}}(f) : 
\bsym{G}_2(x) \to \bsym{G}_2(y) , \]
with formula
\begin{equation} \label{eqn:106}
\opn{Ad}_{\bsym{G}_1 \curvearrowright \bsym{G}_{2}}(f)(a) := 
\bsym{1}_f \circ a \circ \bsym{1}_{f^{-1}} 
\end{equation}
for $a \in \bsym{G}_2(x)$.
It is a simple verification that
$(\bsym{G}_1, \bsym{G}_{2}, 
\opn{Ad}_{\bsym{G}_1 \curvearrowright \bsym{G}_{2}}, \opn{D})$
is a crossed groupoid.
\end{proof}

\begin{rem}
An amusing consequence of the proof above is that the vertical
composition in a $2$-groupoid can be recovered from the horizontal
compositions.
\end{rem}

In view of Proposition \ref{prop:102}, in a $2$-groupoid $\bsym{G}$ 
we can talk about the group of $2$-morphisms 
$\bsym{G}_2(x)$ for any object $x$. There is a feedback homomorphism 
\[ \opn{D} :  \bsym{G}_2(x) \to \bsym{G}_1(x) , \]
and a twisting 
\[ \opn{Ad}(g) =  
\opn{Ad}_{\bsym{G}_1 \curvearrowright \bsym{G}_{2}}(g) :
\bsym{G}_2(x) \to \bsym{G}_2(y)  \]
for any $g : x \to y$.
These satisfy the conditions of a crossed groupoid.

\section{The Deligne $2$-Groupoid}
\label{sec:del-2-grpd}

\begin{dfn}
A DG Lie algebra $\g =  \bigoplus_{i \in \Z}\, \g^i$ is
said to be of {\em quantum type} if $\g^i = 0$ for all $i < -1$. 

A DG Lie algebra
$\til{\g} =  \bigoplus_{i \in \Z}\, \til{\g}^i$ is said to be of
{\em quasi quantum type} if there exists
a quantum type DG Lie algebra $\g$, and a DG Lie algebra quasi-isomorphism 
$\til{\g} \to \g$.
\end{dfn}

\begin{exa} \label{exa:102}
Let $C$ be a commutative $\K$-algebra. The DG Lie algebras 
$\mcal{T}_{\mrm{poly}}(C)$ and $\mcal{D}_{\mrm{poly}}(C)$
that occur in deformation quantization are of quantum type (and hence the
name). 

Let $\g$ be a quantum type DG Lie algebra. Consider the DG Lie algebra 
$\til{\g} := (\opn{L} \circ \opn{C})(\g)$;
this is the bar-cobar construction discussed in Section \ref{sec:L-infty}. 
There is a quasi-isomorphism 
$\zeta_{\g} : \til{\g} \to \g$,
so $\til{\g}$ is of quasi quantum type (but is unbounded in the negative
direction). 
\end{exa}

Suppose $\g$ is a quantum type DG Lie algebra, and $R$ is artinian. Then the 
{\em Deligne $2$-groupoid} of $\m \ot \g$ is defined; see \cite{Ge}. 
In this section we show how this construction can be extended in two ways: 
$\g$ can be of quasi quantum type, and $R$ can be complete (i.e.\ not artinian).

Now let $\g$ be any DG Lie algebra, and $(R, \m)$ any parameter algebra. We have
the set $\opn{MC}(\m \hot \g)$ of MC elements, and the gauge
group $\opn{exp}(\m \hot \g^0)$.

Given $\om \in \opn{MC}(\m \hot \g)$ there is an $R$-bilinear function 
$[-,-]_{\om}$ on $\m \hot \g^{-1}$, whose formula is 
\begin{equation}
[\al_1, \al_2]_{\om} := [ \d_{\om}(\al_1), \al_2 ] , 
\end{equation}
where $\d_{\om} = \d + \opn{ad}(\om)$.

Define  
\begin{equation}
 \a_{\om} := \opn{Coker}(\d_{\om} : \m \hot \g^{-2} \to 
\m \hot \g^{-1}) ,
\end{equation}
so there is an exact sequence of $R$-modules
\begin{equation} \label{eqn:107}
\m \hot \g^{-2} \xar{\d_{\om}} \m \hot \g^{-1} \to \a_{\om} \to 0 .
\end{equation}

\begin{prop} \label{prop:110}
Take any $\om \in \opn{MC}(\m \hot \g)$.
\begin{enumerate}
\item Let $\al \in \m \hot \g^{-1}$ and 
$\be \in \m \hot \g^{-2}$.
Write $\al' := \d_{\om}(\be) \in \m \hot \g^{-1}$.
Then 
\[ [\al, \al']_{\om} , \ [\al', \al]_{\om} \in \d_{\om}(\m \ot \g^{-2}) . \]

\item The induced $R$-bilinear function $[-,-]_{\om}$ on $\a_{\om}$ is a Lie
bracket. Thus $\a_{\om}$ is a Lie algebra.

\item Let $g \in \opn{exp}(\m \hot \g^0)$, and let 
$\om' := \opn{Af}(g)(\om) \in \opn{MC}(\m \hot \g)$. Then 
\[ \opn{Ad}(g) : \a_{\om} \to \a_{\om'} \]
is an isomorphism of Lie algebras.
\end{enumerate}
\end{prop}

\begin{proof}
(1) and (2) are easy direct calculations. (3) is a consequence of Proposition
\ref{prop:100}.
\end{proof}

\begin{prop} \label{prop:103}
Assume either of these two conditions holds\tup{:}
\begin{itemize}
\rmitem{i} $R$ is artinian.  
\rmitem{ii} $\g$ is of quasi quantum type.
\end{itemize}
Then for any $\om \in \opn{MC}(\m \hot \g)$ the 
$R$-module $\a_{\om}$ is $\m$-adically complete. Hence 
$\a_{\om}$ is a pronilpotent Lie algebra.
\end{prop}

\begin{proof}
If $R$ is artinian then any $R$-module is $\m$-adically complete.

Now assume $R$ is not artinian (namely it is a complete noetherian ring). 
If $\g$ is of quantum type then 
$\a_{\om} = \m \hot \g^{-1}$, which is $\m$-adically complete (cf.\ 
\cite[Corollary 3.5]{Ye5}). 

For any DG Lie algebra $\g$ the canonical homomorphism 
\begin{equation} \label{eqn:108}
\tau_{\om} : \a_{\om} \to \what{\a_{\om}} = 
\lim_{\leftarrow i}\, (R_i \ot_R \a_{\om})
\end{equation}
is surjective. Here is the reason: the completion functor $M \mapsto \what{M}$
is not exact (neither right nor left exact), but it does preserve surjections
(see \cite[Proposition 1.2]{Ye5}). Combining this with the
exact sequence (\ref{eqn:107}), and the fact that $\m \hot \g^{-1}$ is
complete, it follows that the homomorphism $\tau_{\om}$ is surjective.

It remains to prove that if there exists a DG Lie algebra quasi-isomorphism
$\phi : \g \to \h$, for some quantum type DG Lie algebra $\h$, then the
homomorphism $\tau_{\om}$ is injective. Since 
$\d_{\om} : \a_{\om} \to \m \hot \g^{0}$ factors through 
$\what{\a_{\om}}$, it follows that 
\[ \opn{Ker}(\tau_{\om}) \subset 
\opn{Ker}(\d_{\om} : \a_{\om} \to \m \hot \g^{0}) = 
\mrm{H}^{-1}(\m \hot \g)_{\om} . \]
Let $\chi := \phi(\om)$. We have a commutative diagram with exact rows 
\[ \UseTips \xymatrix @C=5ex @R=5ex {
0
\ar[r]
&
\mrm{H}^{-1}(\m \hot \g)_{\om}
\ar[r]
\ar[d]_{\mrm{H}^{-1}(\bsym{1} \hot \phi)}
&
\a_{\om}
\ar[r]^(0.35){\d_{\om}}
\ar[d]
&
(\m \hot \g^0)_{\om}
\ar[d]
\\
0
\ar[r]
&
\mrm{H}^{-1}(\m \hot \h)_{\chi}
\ar[r]
&
\a_{\chi}
\ar[r]^(0.35){\d_{\chi}}
&
(\m \hot \h^0)_{\chi} \ .
} \]
Because $\phi$ is a quasi-isomorphism, so is 
\[ \bsym{1} \hot \phi : (\m \hot \g)_{\om} \to 
(\m \hot \h)_{\chi} \]
(we are using Lemma \ref{lem:6}(1)).
Hence the vertical arrow $\mrm{H}^{-1}(\bsym{1} \hot \phi)$ in the diagram is an
isomorphism of $R$-modules. It sends $\opn{Ker}(\tau_{\om})$ bijectively to 
$\opn{Ker}(\tau_{\chi} : \a_{\chi} \to \what{\a_{\chi}})$.
We know that $\a_{\chi}$ is complete, so  $\opn{Ker}(\tau_{\chi}) = 0$.
\end{proof}

\begin{cor} \label{cor:101}
In the situation of Proposition \tup{\ref{prop:103}}, for every 
$\om \in \opn{MC}(\m \hot \g)$
there is a pronilpotent group 
$N_{\om} := \opn{exp}(\a_{\om})$, and a group homomorphism 
\[ \opn{D}_{\om} :=  \opn{exp}(\d_{\om}) : N_{\om} \to 
\opn{exp}(\m \hot \g^0)(\om)  . \]

Given any $g \in \opn{exp}(\m \hot \g^0)$
and $\om \in \opn{MC}(\m \hot \g)$, let 
$\om' := \opn{Af}(g)(\om) \in$ \lb $\opn{MC}(\m \hot \g)$. Then there is a
group isomorphism
\[ \Psi(g) := \opn{exp}(\opn{Ad}(g)) : N_{\om} \iso N_{\om'} , \]
and the diagram 
\[ \UseTips \xymatrix @C=5ex @R=5ex {
N_{\om}
\ar[d]_{\Psi(g)}
\ar[r]^(0.34){\opn{D}_{\om}}
&
\opn{exp}(\m \hot \g^0)(\om)
\ar[d]^{\opn{Ad}(g)}
\\
N_{\om'}
\ar[r]^(0.34){\opn{D}_{\om'}}
&
\opn{exp}(\m \hot \g^0)(\om')
} \]
is commutative.
The isomorphisms $\Psi(g)$ are an action of the group 
$\opn{exp}(\m \hot \g^0)$ on the collection of groups 
$\{ N_{\om} \}_{\om \in \opn{MC}(\m \hot \g)}$.

Moreover, for any $a, a' \in N_{\om}$ we have 
\[ \Psi( \opn{D}_{\om}(a))(a') = 
\opn{Ad}_{N_{\om}}(a)(a') . \]
\end{cor}

\begin{proof}
Combine Propositions \ref{prop:100}, \ref{prop:110} and \ref{prop:103}.
\end{proof}

\begin{rem}
The Lie algebra $\a_{\om}^{\mrm{r}}$ and the group $N_{\om}^{\mrm{r}}$ that
occur in Section 
\ref{sec:red-del} are quotients, respectively, of the Lie algebra $\a_{\om}$ and
the group $N_{\om}$ that occur here. 
\end{rem}

\begin{dfn} \label{dfn:103}
Let $\g$ be a DG Lie algebra and $R$ a parameter algebra. 
Assume either of these two conditions holds:
\begin{itemize}
\rmitem{i} $R$ is artinian.  
\rmitem{ii} $\g$ is of quasi quantum type.
\end{itemize} 
The {\em Deligne $2$-groupoid} $\mbf{Del}^2(\g, R)$
is the transformation $2$-groupoid (see Example \ref{exa:100} and Proposition 
\ref{prop:102}) associated to the action of the group 
$\opn{exp}(\m \hot \g^0)$
on the collection of groups 
$\{ N_{\om} \}_{\om \in \opn{MC}(\m \hot \g)}$.
The feedback $\opn{D}_{\om}$ and the twisting $\Psi(g)$ are specified in
Corollary \ref{cor:101}.
\end{dfn}

\begin{prop}
Consider pairs $(\g, R)$ such that Deligne $2$-groupoid $\mbf{Del}^2(\g, R)$
is defined, i.e.\ either of the two conditions in Definition \tup{\ref{dfn:103}}
holds. 
\begin{enumerate}
\item The Deligne $2$-groupoid $\mbf{Del}^2(\g, R)$ depends functorially on 
$\g$ and $R$. 

\item The reduced Deligne groupoid satisfies
\[ \mbf{Del}^{\mrm{r}}(\g, R) = \pi_1(\mbf{Del}^2(\g, R)) . \]
\end{enumerate}
\end{prop}

\begin{proof}
This is immediate from the constructions.
\end{proof}

\begin{thm} \label{thm:105}
Let $R$ be a parameter algebra, let $\g$ and $\h$ be
DG Lie algebras, and let $\phi : \g \to \h$ be a DG Lie algebra
quasi-isomorphism. Assume either of these two conditions holds:
\begin{itemize}
\rmitem{i} $R$ is artinian.  
\rmitem{ii} $\g$ and $\h$ are of quasi quantum type.
\end{itemize}  
Then the morphism of $2$-groupoids
\[ \mbf{Del}^2(\phi, R) : \mbf{Del}^2(\g, R) \to \mbf{Del}^2(\h, R) \]
is a weak equivalence.
\end{thm}

\begin{proof}
Since 
\[ \pi_0(\mbf{Del}^2(\phi, R)) = \pi_0(\mbf{Del}^{\mrm{r}}(\phi, R)) \]
and 
\[ \pi_1(\mbf{Del}^2(\phi, R), \om) = 
\pi_1(\mbf{Del}^{\mrm{r}}(\phi, R), \om) \]
for all $\om \in \opn{MC}(\m \hot \g)$,
these are bijections by Theorem \ref{thm:2}.

Next, there is a functorial bijection 
\begin{equation} \label{eqn:110}
\begin{aligned}
\exp & : \mrm{H}^{-1}(\m \hot \g)_{\om} \iso 
\\
& \qquad 
\opn{Ker} \bigl( \opn{D}_{\om} : N_{\om} \to \opn{exp}(\m \hot \g^0) \bigr) =
\pi_2(\mbf{Del}^2(\g, R), \om) ;
\end{aligned}
\end{equation}
cf.\ the commutative diagram in the proof of Proposition \ref{prop:103}. 
So 
\[ \pi_2(\mbf{Del}^2(\phi, R), \om) = 
\mrm{H}^{-1}(\bsym{1}_{\m} \hot \phi) \]
is bijective. 
\end{proof}

\begin{rem}
It is possible to define a Deligne $2$-groupoid $\mbf{Del}^2(\g, R)$ even when
both conditions in Definition \ref{dfn:103} fail, by taking 
$N_{\om} := \exp(\what{\a_{\om}})$. 
However the function $\exp$ in (\ref{eqn:110}) might fail to be bijective.
Hence, in the situation of Theorem \ref{thm:105}, the homomorphism 
$\pi_2(\mbf{Del}^2(\phi, R), \om)$ might fail to be bijective.
\end{rem}

\section{$\mrm{L}_{\infty}$ Morphisms and Coalgebras}
\label{sec:L-infty}

We shall use the coalgebra approach to $\mrm{L}_{\infty}$ morphisms, following
\cite{Qu}, \cite[Section 4]{Ko2}, \cite{Hi2}, \cite[Section 3.7]{CKTB} and
\cite[Section 3]{Ye2}. 

Let us denote by $\cat{DGLie}(\K)$ the category of DG Lie algebras over $\K$,
and by $\cat{DGCog}(\K)$ the category of DG unital cocommutative coalgebras
over $\K$. Note that commutativity here is in the graded (or super) sense.
Recall that a unital coalgebra $C$ has a comultiplication 
$\Delta : C \to C \ot C$,
a counit $\ep : C \to \K$ and a unit $1 \in C$. The differential $\d$ has to be
a coderivation of degree $1$. The conditions on the unit are
$\Delta(1) = 1 \ot 1$, $\d(1) = 0$ and $\ep(1) = 1$ in $\K$. 
Morphisms in $\cat{DGCog}(\K)$ are $\K$-linear homomorphisms
$C \to D$ respecting
$\Delta$, $\ep$ and $1$. We write $C^+ := \opn{Ker}(\ep)$.

Let $V = \boplus_{i \in \Z} V^i$ be a graded $\K$-module. The 
symmetric algebra over $\K$ of the graded module $V$ is 
\[ \opn{Sym}(V) = \bigoplus_{i \in \N}\, \opn{Sym}^j(V) . \]
Note that we are in the super-commutative setting, so
$\opn{Sym}^j(V) = \bwedge^j(V[1])[-j]$.
We view $\opn{Sym}(V)$ as a Hopf algebra over $\K$, 
which is commutative and cocommutative. The unit is 
$1 \in \K = \opn{Sym}^0(V)$, and the counit is the projection
$\ep : \opn{Sym}(V) \to \K$.

The Hopf algebra $\opn{Sym}(V)$ is bigraded, with one grading coming from the
grading of $V$, which we call degree. The second grading is called order; by
definition the $j$-th order component of $\opn{Sym}(V)$ is $\opn{Sym}^j(V)$.
Let us write
\[ \opn{Sym}^+(V) := \bigoplus_{i \geq 1}\, \opn{Sym}^j(V) =
\opn{Ker}(\ep) . \]

The projection $\opn{Sym}(V) \to \opn{Sym}^1(V)$ is denoted by $\ln$.
So for an element $c \in \opn{Sym}(V)$, its first order term is $\ln(c)$.
Recall that giving a homomorphism of unital graded coalgebras
$\Psi : \opn{Sym}(V) \to \opn{Sym}(W)$ 
is the same as giving its sequence of Taylor coefficients 
$\{ \pa^j\Psi \}_{j \geq 1}$,
where the $j$-th Taylor coefficient of $\Psi$ is the $\K$-linear function 
\[ \partial^j \Psi := (\ln \circ \Psi)|_{\opn{Sym}^j(V)} : 
\opn{Sym}^j(V) \to W  \]
of degree $0$.

The free graded Lie algebra over the graded $\K$-module $V$ is denoted by \lb
$\opn{FrLie}(V)$. 

There is a functor 
\[ \opn{C} : \cat{DGLie}(\K) \to \cat{DGCog}(\K) \]
called the {\em bar construction}. Given a DG Lie algebra $\g$, the
corresponding DG coalgebra is 
$\opn{C}(\g) := \opn{Sym}(\g[1])$, with a coderivation that encodes both the
differential of $\g$ and its Lie bracket. 
There is another functor 
\[ \opn{L} : \cat{DGCog}(\K) \to \cat{DGLie}(\K)   \]
called the {\em cobar construction}. Given a DG coalgebra $C$, the
corresponding DG Lie algebra is $\opn{FrLie}(C^+[-1])$, with a derivation that
encodes both the differential of $C$ and its comultiplication. 
If $\g \in \cat{DGLie}(\K)$ and $C \in \cat{DGCog}(\K)$, then the set of graded
$\K$-linear homomorphisms $\opn{Hom}^{\mrm{gr}}(C, \g)$ 
is a DG Lie algebra, and there are functorial bijections
\begin{equation} \label{eqn:60}
\opn{Hom}_{\cat{DGLie}(\K)}(\opn{L}(C), \g) \cong
\opn{Hom}_{\cat{DGCog}(\K)}(C, \opn{C}(\g)) \cong
\opn{MC}(\opn{Hom}^{\mrm{gr}}(C, \g)) . 
\end{equation}
Thus the functors $\opn{C}$ and $\opn{L}$ are adjoint.
We denote the adjunction morphisms by
\[ \zeta_{\g} : (\opn{L} \circ \opn{C})(\g) \to \g \]
and 
\[ \th_C : C \to (\opn{C} \circ \opn{L})(C) .  \]

It is known that the functor $\opn{C}$ is faithful. 
Let $\g$ and $\h$ be DG Lie algebras.
By definition, an {\em $\mrm{L}_{\infty}$ morphism} $\Phi : \g \to \h$ 
is a morphism $\opn{C}(\g) \to \opn{C}(\h)$ in $\cat{DGCog}(\K)$.
Let us define $\cat{DGLie}_{\infty}(\K)$ to be the category whose objects are
the DG Lie algebras, and the morphisms are the  $\mrm{L}_{\infty}$ morphisms
between them. Composition of  $\mrm{L}_{\infty}$ morphisms is that of
$\cat{DGCog}(\K)$. Thus we get a full and faithful functor
\[ \opn{C}_{\infty} : \cat{DGLie}_{\infty}(\K) \to \cat{DGCog}(\K) \]
whose restriction to $\cat{DGLie}(\K)$ is $\opn{C}$.

Take an $\mrm{L}_{\infty}$ morphism 
$\Phi : \g \to \h$. Its $i$-th Taylor coefficient
$\pa^i \Phi := \pa^i (\opn{C}_{\infty}(\Phi))$ is a $\K$-linear function 
\[ \pa^i \Phi : \bwedge^i \g \to \h \]
of degree $1 - i$. Writing $\phi_i := \pa^i \Phi$, the sequence of 
functions $\{ \phi_i \}_{i \geq 1}$ satisfies these equations:
\[ \begin{aligned}
& \mrm{d} \bigl( \phi_i(\gamma_1 \wedge \cdots 
\wedge \gamma_i) \bigr) - \sum_{k = 1}^i \pm 
\phi_i \bigl( \gamma_1 \wedge \cdots \wedge 
\mrm{d}(\gamma_k) \wedge \cdots \wedge \gamma_i \bigr) = \\
& \quad {\smfrac{1}{2}} \sum_{\substack{k, l \geq 1 \\ k + l = i}}
{\smfrac{1}{k! \, l!}} \sum_{\sigma \in \mfrak{S}_i} \pm
\bigl[ \phi_k (\gamma_{\sigma(1)} \wedge \cdots \wedge 
\gamma_{\sigma(k)}),
\phi_l (\gamma_{\sigma(k + 1)} \wedge \cdots \wedge 
\gamma_{\sigma(i)}) \bigr] \\
& \qquad +
\sum_{k < l} \pm
\phi_{i-1} \bigl( [\gamma_k, \gamma_l] \wedge
\gamma_{1} \wedge \cdots \gamma_k \hspace{-1em} \diagup \cdots
\gamma_l \hspace{-1em} \diagup 
\cdots \wedge \gamma_{i} \bigr) .
\end{aligned} \]
Here $\gamma_k \in \mfrak{g}$ are homogeneous elements,
$\mfrak{S}_i$ is the permutation group of $\{ 1, \ldots, i \}$,
and the signs depend only on the indices, the permutations and the 
degrees of the elements $\gamma_k$. The
signs are written explicitly in \cite[Section 3.6]{CKTB}. 
Conversely, any sequence $\{ \phi_i \}_{i \geq 1}$ of homomorphisms satisfying
these equations determines an $\mrm{L}_{\infty}$ morphism.

Let $\Phi : \g \to \h$ be  an $\mrm{L}_{\infty}$ morphism.
The first Taylor coefficient  $\pa^1 \Phi : \g \to \h$ is a homomorphism of
complexes of $\K$-modules (forgetting the Lie brackets). 
If $\pa^i \Phi = 0$ for all $i \geq 2$, then 
$\pa^1 \Phi$ is a DG Lie algebra homomorphism. 
If $\pa^1 \Phi$ is a quasi-isomorphism, then $\Phi$ is called an 
{\em $\mrm{L}_{\infty}$ quasi-isomorphism}.

\begin{lem} \label{lem:20}
Let $\g \in \cat{DGLie}(\K)$, and define
$C := \opn{C}(\g)$, $\til{\g} := \opn{L}(C)$ and $\til{C} := \opn{C}(\til{g})$.
Consider the coalgebra homomorphisms 
$\th_C : C \to \til{C}$ and $\opn{C}(\zeta_{\g}) : \til{C} \to C$. 
Then 
\[ \opn{C}(\zeta_{\g}) \circ \th_C = \bsym{1}_C , \]
the identity automorphism of $C$.
\end{lem}

\begin{proof}
This is true for any pair of adjoint functors; see 
\cite[Section IV.1 Theorem 1]{Ma}.
\end{proof}

\begin{lem} \label{lem:21}
Let $\Phi : \g \to \h$ be an $\mrm{L}_{\infty}$ quasi-isomorphism. Then 
\[ (\opn{L} \circ \opn{C}_{\infty})(\Phi) : 
(\opn{L} \circ \opn{C})(\g) \to (\opn{L} \circ \opn{C})(\h) \]
is a DG Lie algebra quasi-isomorphism.
\end{lem}

\begin{proof}
Let us write
$C := \opn{C}(\g)$, $D := \opn{C}(\h)$ and $\Psi := \opn{C}_{\infty}(\Phi)$. 
Put on $C$ the ascending filtration 
$F_j C := \boplus_{k = 0}^j \opn{Sym}^k(\g[1])$; 
so that $\{ F_j C \}_{j \geq 0}$ is an admissible coalgebra filtration, in the
sense of \cite[Definition 4.4.1]{Hi2}. Likewise there is an 
admissible coalgebra filtration $\{ F_j D \}_{j \geq 0}$ on $D$.
According to step 2 of the proof of \cite[Proposition 4.4.3]{Hi2}, the
coalgebra homomorphism $\Psi$ is a filtered quasi-isomorphism. 
Hence by \cite[Proposition 4.4.4]{Hi2} the DG Lie algebra homomorphism 
$(\opn{L} \circ \opn{C}_{\infty})(\Phi) = \opn{L}(\Psi)$
is a quasi-isomorphism.
\end{proof}

\section{$\mrm{L}_{\infty}$ Quasi-isomorphisms between Pronilpotent Algebras}
\label{sec:L-infty-cplt}

Let $(R, \m)$ be a parameter $\K$-algebra, and let
$\Phi : \g \to \h$ be an $\mrm{L}_{\infty}$ morphism between DG Lie algebras. 
Then $\Phi$ extends uniquely to an $R$-multilinear $\mrm{L}_{\infty}$ morphism 
\[ \Phi_{R} : R \hot \g \to R \hot \h , \]
whose $i$-th Taylor coefficient
\[ \partial^i \Phi_{R} :
\underset{i}{\underbrace{(R \hot \g) \times \cdots \times 
(R \hot \g)}} \to R \hot \h \]
is the $R$-multilinear homogeneous extension of 
$\partial^i \Phi$. See \cite[Section 3]{Ye2} for more details. 

\begin{dfn} \label{dfn:6}
Let $\Phi : \g \to \h$ be an $\mrm{L}_{\infty}$ morphism, and let 
$(R, \m)$ be a parameter algebra. For an element 
$\om \in \m \hot \g^1$ we define
\[ \opn{MC}(\Phi, R)(\om) := 
\sum_{i \geq 1} \, \smfrac{1}{i!} (\partial^i \Phi_{R})
(\underset{i}{\underbrace{\om, \ldots, \om}}) \in \m \hot \h^1 . \]
\end{dfn}

Note that the sum above converges in the $\m$-adic topology of 
$\m \hot \h^1$, since 
\[ (\partial^i \Phi_{R})(\om, \ldots, \om)
\in \m^i \hot \h^1 = \m^{i-1} \cdot (\m \hot \h^1) . \]

\begin{prop} \label{prop:4}
Let 
$\om \in \opn{MC}(\g, R) = \opn{MC}(\m \hot \g)$.
Then the element
$\opn{MC}(\Phi, R)(\om)$ belongs to $\opn{MC}(\h, R)$.
Thus we get a function 
\[ \opn{MC}(\Phi, R) : \opn{MC}(\g, R) \to \opn{MC}(\h, R) . \]
\end{prop}

\begin{proof}
Let $R_j := R / \m^{j+1}$, and let $p_j : R \to R_j$ be the projection.
According to \cite[Theorem 3.21]{Ye2}, which refers to the artinian case, we
have
\[ \opn{MC}(\Phi, R_j)(p_j(\om)) \in \opn{MC}(\h , R_j) \]
for every $j$. And by Lemma \ref{lem:6}(2) we know that 
\[ \opn{MC}(\h, R) \cong 
\lim_{\leftarrow j}\, \opn{MC}(\h, R_j) . \]
\end{proof}

Let $\Om_{\K[t]} = \K[t] \oplus \Om^1_{\K[t]}$ be the algebra of polynomial
differential forms
in the variable $t$ (relative to $\K$). This is a commutative DG
algebra. For $\la \in \K$ there is a DG algebra homomorphism
$\sg_{\la} : \Om_{\K[t]} \to \K$, namely $t \mapsto \la$.
There is an induced homomorphism of DG Lie algebras 
\begin{equation} \label{eqn:2}
\sg_{\la}  :  \m \hatotimes \Om_{\K[t]} \hatotimes \g \to 
\m \hatotimes \g ,
\end{equation}
and an induced function
\begin{equation} \label{eqn:61}
\ol{\opn{MC}}(\sg_{\la}) : 
\ol{\opn{MC}} ( \m \hatotimes \Om_{\K[t]} \hatotimes \g) 
\to \ol{\opn{MC}}(\m \hatotimes \g) 
\end{equation}

We shall often think of elements of $\Om_{\K[t]}$ as 
``functions of $t$'', as in Section \ref{sec:facts}. Given elements
$f(t) \in \Om_{\K[t]}$ and $\la \in \K$,
we shall use the substitution notation  
$f(\la) := \sg_{\la} (f(t)) \in \K$.

\begin{lem} \label{lem:16}
Here $\lambda = 0, 1$. 
\begin{enumerate}
\item The homomorphisms $\sg_0, \sg_1$ in formula \tup{(\ref{eqn:2})}
are quasi-iso\-morphisms.

\item The functions $\ol{\opn{MC}}(\sg_0)$ and $\ol{\opn{MC}}(\sg_1)$
in formula \tup{(\ref{eqn:61})} are bijections.

\item The bijections 
$\ol{\opn{MC}}(\sg_0)$ and $\ol{\opn{MC}}(\sg_1)$ are equal.
\end{enumerate} 
\end{lem}

\begin{proof}
(1) This is because the homomorphisms $\sg_i : \Om_{\K[t]} \to \K$ are
homotopy equivalences (of complexes of $\K$-modules).

\medskip \noindent
(2) This is by item (1) and Theorem \ref{thm:2}.

\medskip \noindent
(3) Note that the inclusion
\[ \phi : \m \hot \g \to \m \hatotimes \Om_{\K[t]} \hatotimes \g \]
is also a quasi-isomorphism, and that 
$\sg_0 \circ \phi = \sg_1 \circ \phi$ is the identity automorphism of 
$\m \hot \g$. Hence 
\[ \ol{\opn{MC}}(\sg_0) \circ \ol{\opn{MC}}(\phi) = 
\ol{\opn{MC}}(\sg_1) \circ \ol{\opn{MC}}(\phi) , \]
and canceling the bijection $\ol{\opn{MC}}(\phi)$ we obtain our result.
\end{proof}

\begin{lem} \label{lem:17}
Let $\om_0, \om_1 \in \opn{MC}(\m \hatotimes \g)$. The following two
conditions are equivalent:
\begin{enumerate}
\rmitem{i} There exists an element 
$g \in \opn{exp}(\m \hatotimes \g^0)$ such that 
\[ \opn{Af}(g)(\om_0) = \om_1 . \]
\rmitem{ii} There exists an element 
\[ \om(t) \in 
\opn{MC}( \m \hatotimes \Om_{\K[t]} \hatotimes \g) \]
such that 
$\om(0) = \om_0$ and $\om(1) = \om_1$. 
\end{enumerate}
\end{lem}

\begin{proof}
(i) $\Rightarrow$ (ii): Consider the elements $\ga := \log(g) \in \m \hot
\g^0$, 
\[ g(t) := \exp(t \ga) \in \exp(\m \hot (\Om_{\K[t]} \ot \g)^0) \]
and 
\[ \om(t) :=  \opn{Af}(g(t))(\om_0) \in 
\opn{MC}(\m \hot \Om_{\K[t]} \hot \g) . \]
Then 
\[ \om(0) = \opn{Af}(g(0))(\om_0) = \opn{Af}(1)(\om_0) = \om_0 \]
and
\[ \om(1) = \opn{Af}(g(1))(\om_0) =
\opn{Af}(g)(\om_0) = \om_1 . \]

\medskip \noindent
(ii) $\Rightarrow$ (i): Let us write $[\om_i]$ for the classes in 
$\ol{\opn{MC}}(\m \hatotimes \g)$.
We know that 
\[ \ol{\opn{MC}}(\sg_i)([\om(t)]) = [\om_i] \]
for $i = 0, 1$. By Lemma \ref{lem:16}(3) we know that 
$\ol{\opn{MC}}(\sg_0) = \ol{\opn{MC}}(\sg_1)$.
Therefore $[\om_0] = [\om_1]$; and by definition this says that 
$\opn{Af}(g)(\om_0) = \om_1$ for some $g$. 
\end{proof}

\begin{prop} \label{prop:5}
Let $\Phi : \g \to \h$ be an $\mrm{L}_{\infty}$ morphism.
The function 
\[ \opn{MC}(\Phi, R) : \opn{MC}(\g, R) \to \opn{MC}(\h, R) \]
respects gauge equivalences. Therefore there is an induced function 
\[ \ol{\opn{MC}}(\Phi, R) : \ol{\opn{MC}}(\g, R) \to 
\ol{\opn{MC}}(\h, R) . \]
\end{prop}

\begin{proof}
Let $A:= R \hot \Om_{\K[t]}$, which is a commutative DG algebra. There
are induced homomorphisms $\sg_i : A \to R$.
And there is an induced $A$-multilinear $\mrm{L}_{\infty}$ morphism 
\[ \Phi_{A} : \m \hot \Om_{\K[t]} \hot \g \to 
\m \hot \Om_{\K[t]} \hot \h  \]
(cf.\ \cite[Section 3]{Ye2}).
Let us write 
$\opn{MC}(\g, A) := \opn{MC}(\m \hot \Om_{\K[t]} \hot \g)$ etc.
There is a function $\opn{MC}(\Phi, A)$ whose formula is like in Definition
\ref{dfn:6}. Since the $\mrm{L}_{\infty}$ morphisms $\Phi_R$ and
$\Phi_A$ are induced from $\Phi$,
the diagram of functions 
\[ \UseTips \xymatrix @C=11ex @R=6ex {
\opn{MC}(\g, A) 
\ar[r]^{\opn{MC}(\Phi, A)}
\ar[d]_{\opn{MC}(\sg_i, R)}
&
\opn{MC}(\h, A) 
\ar[d]^{\opn{MC}(\sg_i, R)}
\\
\opn{MC}(\g, R) 
\ar[r]^{\opn{MC}(\Phi, R)}
&
\opn{MC}(\h, R) 
} \]
is commutative, for $i = 0, 1$.
Now suppose $\om_0, \om_1 \in \opn{MC}(\g, R)$ are gauge equivalent. By Lemma
\ref{lem:17} there is an element $\om(t) \in \opn{MC}(\g, A)$ such that 
$\om(i) = \om_i$. Define
$\chi_i := \opn{MC}(\Phi, R)(\om_i)$ and 
$\chi(t) := \opn{MC}(\Phi, A)(\om(t))$. 
Because the diagram above is commutative, we have
$\chi(i) = \chi_i$. Using Lemma \ref{lem:17} again we conclude that 
$\chi_0$ and $\chi_1$ are gauge equivalent.
\end{proof}

Let $\Psi : C \to D$ be the morphism in $\cat{DGCog}(\K)$ gotten by applying
the functor $\opn{C}_{\infty}$ to the $\opn{L}_{\infty}$ morphism 
$\Phi : \g  \to \h$. And let 
$\til{\Psi} : \til{C} \to \til{D}$ be the morphism gotten by applying
the functor $\opn{C} \circ \opn{L}$ to $\Psi : C \to D$. 
Since $\th_-$ is a natural transformation, we get a
commutative diagram 
\begin{equation} \label{eqn:62}
\UseTips \xymatrix @C=7ex @R=5ex {
C
\ar[r]^{\Psi}
\ar[d]_{\th_C}
&
D
\ar[d]^{\th_D}
\\
\til{C}
\ar[r]^{\til{\Psi}}
&
\til{D}
} 
\end{equation}
in $\cat{DGCog}(\K)$.
There is a corresponding commutative diagram 
\begin{equation} \label{eqn:63}
\UseTips \xymatrix @C=7ex @R=5ex {
\g
\ar[r]^{\Phi}
\ar[d]_{\th_{\g}}
&
\h
\ar[d]^{\th_{\h}}
\\
\til{\g}
\ar[r]^{\til{\Phi}}
&
\til{\h}
} 
\end{equation}
in $\cat{DGLie}_{\infty}(\K)$. Namely 
$\til{\g} = (\opn{L} \circ \opn{C})(\g)$,
$\til{\h} = (\opn{L} \circ \opn{C})(\h)$, 
and the full faithful functor $\opn{C}_{\infty}$ sends the 
diagram (\ref{eqn:63}) to the diagram (\ref{eqn:62}).
Note that by the proof of Lemma \ref{lem:20}, the Taylor coefficients
$\pa^j \th_C$ are nonzero for all $j$; so the corresponding morphism 
$\th_{\g} : \g \to \til{\g}$ in $\cat{DGLie}_{\infty}(\K)$ is not a DG Lie
algebra homomorphism. The same for $\th_{\h}$.

\begin{lem} \label{lem:18}
The diagram of functions
\begin{equation} \label{eqn:65}
\UseTips \xymatrix @C=11ex @R=6ex {
\opn{MC}(\g, R)
\ar[r]^{\opn{MC}(\Phi, R)}
\ar[d]_{\opn{MC}(\th_{\g}, R)}
&
\opn{MC}(\h, R)
\ar[d]^{\opn{MC}(\th_{\h}, R)}
\\
\opn{MC}(\til{\g}, R)
\ar[r]^{\opn{MC}(\til{\Phi}, R)}
&
\opn{MC}(\til{\h}, R)
} 
\end{equation}
is commutative.
\end{lem}

\begin{proof}
Because of Lemma \ref{lem:6}(2) we can assume that $R$ is artinian.
Consider the commutative diagram of DG coalgebras over $R$
\begin{equation} \label{eqn:67}
\UseTips \xymatrix @C=11ex @R=6ex {
R \ot C
\ar[r]^{\Psi_R}
\ar[d]_{\th_{C, R}}
&
R \ot D
\ar[d]^{\th_{D, R}}
\\
R \ot \til{C}
\ar[r]^{\til{\Psi}_R}
&
R \ot \til{D}
} 
\end{equation}
induced from (\ref{eqn:62}) by tensoring with $R$. 

Take any $\om \in \m \ot \g^1 \subset R \ot C$, and define 
\[ e := \exp(\om) = \sum_{i \geq 0} \, \smfrac{1}{i!} \om^i \in R \ot C . \]
According to \cite[Lemma 3.18]{Ye2} we have
\[ \begin{aligned}
& \bigl( \opn{MC}(\th_{\h}, R) \circ \opn{MC}(\Phi, R) \bigr)(\om) =
\log \bigl( (\th_{D, R} \circ \Psi_R)(e) \bigr) = \\
& \qquad \log \bigl( (\til{\Psi}_R \circ \th_{C, R})(e) \bigr) =
\bigl( \opn{MC}(\til{\Phi}, R) \circ \opn{MC}(\th_{\g}, R) \bigr)(\om) . 
\end{aligned} \]
This proves commutativity of the diagram.
\end{proof}

\begin{thm} \label{thm:3}
Let $\g$ and $\h$ be DG Lie algebras, 
let $\Phi : \g \to \h$ be an $\mrm{L}_{\infty}$ quasi-isomorphism, and
let $R$ be a parameter algebra, all over the field $\K$. 
Then the function 
\[ \ol{\opn{MC}}(\Phi, R) : \ol{\opn{MC}}(\g, R) \to 
\ol{\opn{MC}}(\h, R) \]
\tup{(}see Proposition \tup{\ref{prop:5})} is bijective. 
\end{thm}

The idea for the proof was suggested to us by Van den Bergh.

\begin{proof}
By Lemma \ref{lem:21} the DG Lie algebra homomorphism 
$\til{\Psi} : \til{\g} \to \til{\h}$ is a quasi-isomorphism.
Therefore by Theorem \ref{thm:2} the function 
\[  \ol{\opn{MC}}(\til{\Phi}, R) : \ol{\opn{MC}}(\til{\g}, R) \to 
\ol{\opn{MC}}(\til{\h}, R) \]
is bijective.

Next, by \cite[Proposition 4.4.3(1)]{Hi2} the DG Lie algebra homomorphism
$\zeta_{\g} : \til{\g} \to \g$ is a  quasi-isomorphism.
Again using Theorem \ref{thm:2} we conclude that the function 
\[  \ol{\opn{MC}}(\zeta_{\g}, R) : \ol{\opn{MC}}(\til{\g}, R) \to 
\ol{\opn{MC}}(\g, R) \]
is bijective. On the other hand, by Lemma \ref{lem:20}, with the arguments in
the proof of Lemma \ref{lem:18}, we see that 
\[ \opn{MC}(\zeta_{\g}, R) \circ \opn{MC}(\th_{\g}, R) = 
\bsym{1}_{\opn{MC}(\g, R)} . \]
Therefore 
\[ \ol{\opn{MC}}(\zeta_{\g}, R) \circ \ol{\opn{MC}}(\th_{\g}, R) = 
\bsym{1}_{\ol{\opn{MC}}(\g, R)} . \]
Because $\ol{\opn{MC}}(\zeta_{\g}, R)$ is a bijection, the same is true for the
function $\ol{\opn{MC}}(\th_{\g}, R)$. 
The same line of reasoning says that 
$\ol{\opn{MC}}(\th_{\h}, R)$ is bijective. 

Finally consider the commutative diagram of functions
\[ \UseTips \xymatrix @C=11ex @R=6ex {
\ol{\opn{MC}}(\g, R)
\ar[r]^{\ol{\opn{MC}}(\Phi, R)}
\ar[d]_{\ol{\opn{MC}}(\th_{\g}, R)}
&
\ol{\opn{MC}}(\h, R)
\ar[d]^{\ol{\opn{MC}}(\th_{\h}, R)}
\\
\ol{\opn{MC}}(\til{\g}, R)
\ar[r]^{\ol{\opn{MC}}(\til{\Phi}, R)}
&
\ol{\opn{MC}}(\til{\h}, R)
} \]
induced from (\ref{eqn:65}). 
We know that three of the arrows are bijective; so the fourth arrow, namely 
$\ol{\opn{MC}}(\Phi, R)$, is also bijective.
\end{proof}

\begin{rem} \label{rem:1}
Lemma 3.5 in our earlier paper \cite{Ye1} says that the canonical function 
\[ \ol{\opn{MC}}(\g, R) \to \lim_{\leftarrow j} \,
\ol{\opn{MC}}(\g, R / \m^j) \]
is bijective. The proof of this Lemma
(which is actually omitted from the paper) is incorrect. Moreover, we
suspect that statement itself is false. The hidden assumption was that 
the gauge group $\opn{G}(\g, R)$ acts on the set $\opn{MC}(\g, R)$
with {\em closed} orbits. 

This lemma is used to deduce \cite[Corollary 3.10]{Ye1} (incorrectly) 
from the nilpotent case. The correction is to replace 
\cite[Corollary 3.10]{Ye1} with Theorem \ref{thm:3} above.
The same correction pertains also to \cite[formula (12.2)]{Ye4}.
\end{rem}

\begin{rem}
This is a good place to correct a typographical error (repeated several
times) in \cite[Section 3]{Ye2}. In Lemmas 3.14 and 3.18 of op.\ cit., instead
of ``$\m \g [1]$'' the correct expression should be ``$\m \g^1$'' or
``$(\m \g [1])^0$''.

Let us also mention that a ``colocal coalgebra homomorphism'', in the sense of 
\cite[Definition 3.3]{Ye2}, is the same as a ``unital coalgebra homomorphism''.
\end{rem}


\end{document}